\documentclass[reqno,10pt]{amsart}

\usepackage{amsmath, amsthm, amsfonts, amssymb}
\usepackage{pdfsync}
\usepackage{hyperref}
\usepackage{mathrsfs}  
\usepackage{mathtools}
\usepackage{mdframed}
\usepackage{tikz}
\numberwithin{equation}{section}
\definecolor{MLU}{rgb}{0,0.455,0.345}
\definecolor{hg}{rgb}{0.25,0.6,0}

\newtheorem{theorem}{Theorem}[section]
\newtheorem{prop}[theorem]{Proposition}

\newtheorem{cor}[theorem]{Corollary}

\newtheorem{examples}[theorem]{Examples}
\newtheorem{remark}[theorem]{Remark}
\newtheorem{rmk}[theorem]{Remarks}

\newenvironment{subproof}[1][\proofname]
{%
  \begin{proof}[#1]%
}
{%
  \end{proof}%
}

\def\M{\mathsf{M}}
\def\p{\mathsf{p}}
\def\q{\mathsf{q}}
\def\a{\mathsf{a}}

\def\ev{\mathsf{C}}
\def\cH{\mathcal{H}}

\def\Uk{\mathsf{O}_{\kappa}}

\def\B{\mathbb{B}}
\def\R{\mathbb{R}}
\def\C{\mathbb{C}}
\def\N{\mathbb{N}}
\def\A{\mathfrak{A}}
\def\K{\frak{K}}

\def\ez{\mathbb{E}_{0,\mu}}
\def\ef{\mathbb{E}_{1,\mu}}

\def\L{\mathcal{L}}
\def\Lis{\mathcal{L}{\rm{is}}}
\def\F{\mathfrak{F}}

\def\B{\mathbb{B}}
\def\Bm{\mathbb{B}^m}

\def\cA{\mathcal{A}}

\def\kb{\varphi^{\ast}_{\kappa}}
\def\kf{\psi^{\ast}_{\kappa}}

\def\pk{\pi_\kappa}

\def\Re{\mathcal{R}}
\def\Rek{\mathcal{R}_{\kappa}}
\def\Rc{\mathcal{R}^c}
\def\Rck{\mathcal{R}^c_{\kappa}}

\def\supp{{\rm{supp}}}

\def\uf{{\rm{unif}}}

\def\Rk{\mathbb{R}^m_\kappa}

\begin{document}

\title[SDF and WF for uniformly regular surfaces]{The surface diffusion and  the Willmore flow for 
uniformly regular hypersurfaces}

\author[J. LeCrone]{Jeremy LeCrone}
\address{Department of Mathematics \& Computer Science, 
University of Richmond, 
Richmond, VA 23173, USA}
\email{jlecrone@richmond.edu}

\author[Y. Shao]{Yuanzhen Shao}
\address{Department of Mathematical Sciences,
         Georgia Southern University, 
         Statesboro, GA 30460, USA}
\email{yshao@georgiasouthern.edu}

\author[G. Simonett]{Gieri Simonett}
\address{Department of Mathematics, 
Vanderbilt University, 
Nashville, TN 37240, USA}
\email{gieri.simonett@vanderbilt.edu}

\thanks{This work was supported by a grant from the Simons Foundation (\#426729, Gieri Simonett).}


\begin{abstract}
We consider the surface diffusion and Willmore flows acting on a general class of (possibly non--compact)
hypersurfaces parameterized over a uniformly regular reference manifold possessing a tubular 
neighborhood with uniform radius. 
The surface diffusion and Willmore flows each give rise to a fourth--order quasilinear parabolic
equation with nonlinear terms satisfying a specific singular structure.
We establish well--posedness of both flows for initial surfaces that are 
 $C^{1+\alpha}$--regular and
parameterized over a uniformly regular hypersurface.
For the Willmore flow, we also show long--term existence
for initial surfaces which are $C^{1+\alpha}$--close to a sphere, and we prove
that these solutions become spherical as time goes to infinity.
\end{abstract}
\maketitle

\section{Introduction}\label{Section 1}
The surface diffusion and Willmore flows are geometric evolution equations
that describe the motion of hypersurfaces in Euclidean space 
(or, more generally, in an ambient Riemannian manifold).
The normal velocity of evolving surfaces is determined by purely geometric quantities.
For both flows, the mean curvature is involved in the evolution equations, 
while the Willmore flow additionally depends upon Gauss curvature.

These flows have been studied by several authors for compact (closed) hypersurfaces. 
In this setting, existence, regularity, and qualitative behavior of solutions have been analyzed
in \cite{EscMaySim98, EM10, LeCroneSim13, McWe11, Shao15, Whe11, Whe12}
for the surface diffusion flow, 
and in \cite{Bla09, KS01, KS02, KS04, MaSi02, MaSi03, McWh16, Shao13, Sim01}
for the Willmore flow, to mention just a few publications.

In this paper, we consider {uniformly regular hypersurfaces}.
It should be emphasized that these surfaces may be non-compact.
The concept of uniformly regular Riemannian manifolds was introduced by Amann~\cite{Ama13,Ama12}
and it contains  the class of compact Riemannian manifolds as a special case.
The study of geometric flows on non--compact manifolds is an active
research topic, both from the point of view of PDE theory and in relation to its applications in 
geometry and topology. 
To the best of our knowledge, the current literature on the surface diffusion and Willmore flows for
non--compact manifolds all concern surfaces defined over an infinite cylinder or 
entire graphs over $\R^m$, or the Willmore flow 
with small initial energy, cf. \cite{Asai13, KL12, KS01, LeCroneSim16, LeCroneSim18}.
Our work generalizes the study of these two flows to a larger class of manifolds.

In our main result we establish well--posedness for initial surfaces that are  $C^{1+\alpha}$--regular and
parameterized over a uniformly regular hypersurface.
Moreover, we show that solutions instantaneously regularize and become smooth,
and even analytic in case $\Sigma$ is analytic.
In order to obtain our results, we show that the pertinent underlying evolution equations can
be formulated as parabolic quasilinear equations of fourth order over the reference surface $\Sigma$.
Our analysis relies on the theory of continuous maximal regularity and the results and techniques 
developed in \cite{LeCroneSim18, Shao15,ShaoSim14}.

The results in Theorem~\ref{S3: Thm SDF well-posedness} and 
Theorem~\ref{S5: Thm WF well-posedness} are new.
However, we note that in case $\Sigma$ is an infinitely long cylinder embedded in $\R^3$, 
an analogous result to Theorem~\ref{S3: Thm SDF well-posedness} was obtained in 
\cite{LeCroneSim18} for the surface diffusion flow.

For the Willmore flow, Theorem~\ref {S5: Thm WF well-posedness} is also new 
even if $\Sigma$ is a compact (smooth, closed) surface.
Previous results impose more regularity on the initial surface, for instance $C^{2+\alpha}$ in \cite{Sim01}.

Theorem~\ref{S5: Thm WF stability}, where global existence and convergence to a sphere is shown for
surfaces that are $C^{1+\alpha}$--close to a sphere, also seems to be new.
A corresponding result was obtained in \cite{Sim01} for surfaces close to a sphere in the 
$C^{2+\alpha}$--topology.
The authors in \cite{KS02} showed the existence of a lower bound on the lifespan of a smooth 
solution, which depends only on how much the curvature of the initial surface is concentrated in space.
In \cite{KS01, KS04}, the authors proved convergence to round spheres under suitable smallness
assumptions on the total energy of the surface. 
Here we note that the energy used in \cite{KS01, KS04} involves second--order derivatives,
whereas we only need smallness in the $C^{1+\alpha}$--topology.
In particular, we obtain global existence and convergence for non--convex initial surfaces.
 
\noindent The organization of the paper is as follows:

In Sections 2.1 and 2.2, we introduce the concept of uniformly regular manifolds and define the 
function spaces used in this paper. 
In Sections 2.3 and 2.4, we review continuous maximal regularity theory and its applications to 
quasilinear parabolic equations with singular nonlinearity. 
These results form the theoretic basis for the study of the surface diffusion and Willmore flows.

In Section 3, we introduce the concept of uniformly regular hypersurfaces with a uniform 
tubular neighborhood (called (URT)--hypersurfaces) and work out several examples.
We utilize these concepts to parameterize the evolving hypersurfaces driven by surface diffusion 
and Willmore flows as normal graphs over a (URT)-reference hypersurface.

In Section 4, we establish our main results regarding existence, uniqueness, regularity, and semiflow
properties for solutions to the surface diffusion flow over (URT)--hypersurfaces in $\R^{m+1}$.
In Section 5, we likewise establish well--posedness properties for solutions to the Willmore
flow over (URT)--hypersurfaces in $\R^3$. 
Additionally, we show stability of Euclidean spheres under perturbations in 
the $C^{1+\alpha}$--topology.

We conclude the paper with an appendix where we state and prove some additional properties of 
normal graphs over (URT)-hypersurfaces.

\medskip
\noindent\textbf{Notation:} 
For  two Banach spaces $X$ and $Y$, $X\doteq Y$ means that they are equal in the 
sense of equivalent norms. $\L(X,Y)$ denotes the set of all bounded linear 
maps from $X$ to $Y$ and $\Lis(X,Y)$ is the subset of $\L(X,Y)$ consisting of 
all bounded linear isomorphisms from $X$ to $Y$. 
For $x\in X$, $\B_X(x,r)$ denotes the (open) ball in $X$ with radius $r$ and center $x$.
We sometimes write $\B(x,r)$, in lieu of $\B_X(x,r)$, in case the setting is clear,
and we write $\B^m(x,r)$ when $X = \R^m$.
We denote by $g_m$ the Euclidean metric in $\R^m$. Given an embedded hypersurface 
$\Sigma$ in $\R^m$, $g_m|_\Sigma$ means the metric on $\Sigma$ induced by $g_m$.
Finally, we set $\N_0= \N\cup\{0\}$.


\section{Preliminaries}\label{Section 2}

\subsection{Uniformly regular manifolds}\label{Section 2.1}
The concept of {\em uniformly regular (Riemannian) manifolds} was introduced 
by H.~Amann in \cite{Ama13} and \cite{Ama12}.
Loosely speaking, an $m$--dimensional Riemannian manifold $(\M,g)$ is uniformly regular if its
differentiable structure is induced by an atlas such that all its local patches are of approximately the same size,
all derivatives of the transition maps are bounded, and the pull-back metric of $g$ in every local coordinate is
comparable to the Euclidean metric $g_m$.

We will now state some structural properties of {uniformly regular manifolds}
which will be used in the analysis of the the surface diffusion  flow and  the Willmore flow
in subsequent sections.

An oriented $C^{\infty}$--manifold $(\M,g)$ of dimension $m$ and without 
boundary is {\em uniformly regular} if it admits  an orientation-preserving atlas 
$\A:=\{( \Uk ,\varphi_\kappa): \kappa\in \K\}$, with a countable index set $\K$, 
satisfying the following conditions.

\begin{itemize}
\item[(R1)] There exists $K \in \N$ such that any intersection of more than $K$ 
	coordinate patches is empty.
\item[(R2)] $\varphi_{\kappa}( \Uk )=\Bm,$ where $\Bm$ is the unit Euclidean ball 
	centered at the origin in $\R^m$.  Moreover, $\A$ is uniformly shrinkable; by which we mean
	that there exists some $r\in (0,1)$ such that $\{\psi_\kappa(r \Bm): \kappa\in \K\}$ forms 
	a cover for $\M$, where $\psi_{\kappa}:=\varphi_{\kappa}^{-1}$.
\item[(R3)] $\|\varphi_\eta \circ \psi_\kappa\|_{k,\infty}\leq c(k)$ for all $\kappa\in\K$, 
	$k\in \N_0$ and $\eta\in \K$ such that $\mathsf{O}_\eta \cap\Uk\neq\emptyset$.
\item[(R4)] $\| \psi_\kappa^* g\|_{k,\infty}\leq c(k)$ for all $\kappa\in\K$ and $k\in\N_0$.
\item[(R5)] $\psi_\kappa^* g \sim g_m $ for all $\kappa\in \K$. 
	Here $g_m$ is the Euclidean metric in $\R^m$ and $\psi_\kappa^* g$ denotes the pull-back 
	metric of $g$ by $\psi_\kappa$.
\end{itemize}
Here (R5) means that there exists some number $c\geq 1$  such that
$$
(1/c) |\xi|^2\leq \psi_\kappa^* g(x)(\xi,\xi) \leq c|\xi|^2 ,\quad x\in \Bm,\;\xi\in \R^m, \;\kappa\in\K. 
$$
Given an open subset $U\subset \R^m$, a Banach space $X$, and a mapping  $u:U\to X$,
$$
\|u\|_{k,\infty}:=\max_{|\alpha|\leq k} \|\partial^\alpha u\|_\infty
$$
is the norm of the space $BC^k(U, X)$, which consists of all functions $u\in C^k(U,X)$ such that $\|u\|_{k,\infty}<\infty$.

Any uniformly regular manifold $(\M,g)$ possesses a \emph{localization system subordinate to} $\A$, 
by which we mean a family 
$\{(\pk ,\zeta_\kappa):  \kappa \in \K\}$ satisfying:
\begin{itemize}
\item[(L1)] $ \pk \in \mathcal{D}( \Uk ,[0,1])$ and $\{\pi_\kappa^2: \kappa \in \K \}$ is a partition of unity subordinate to 
            the cover $\{\Uk : \kappa\in\K\}$.     
\item[(L2)] $\zeta_\kappa := \kb \zeta$ with $\zeta \in\mathcal{D}(\Bm,[0,1])$ satisfying 
            $\zeta|_{\supp(\kf \pk )}\equiv 1$, $\kappa \in \K$.
\item[(L3)] $\|\psi_{\kappa}^{\ast} \pk \|_{k,\infty} + \|  \zeta  \|_{k,\infty}  \leq c(k) $, for $\kappa \in \K$, $k\in \N_0$.
\end{itemize}

\medskip
Given $k\in \N\cup \{\omega\}$, the concept of $C^k$--{\em uniformly regular manifold} 
is defined by modifying (R3), (R4), (L1)-(L3) in an obvious way, where $\omega$ is the symbol for 
real analyticity.
\goodbreak
\begin{remark}
In \cite{DisShaoSim}, the  authors showed that a $C^{\infty}$--manifold  without boundary is {uniformly regular} iff it is of bounded geometry, i.e. it is geodesically complete, of positive injectivity radius and all covariant derivatives of the curvature tensor are bounded.
In particular, every compact manifold without boundary is {uniformly regular} and the manifolds considered in \cite{LeCroneSim13, LeCroneSim16} are all {uniformly regular}.
\end{remark}
Given $\sigma,\tau\in\N_0$, we define the $(\sigma,\tau)$--tensor bundle of $\M$ as
$$
T^\sigma_\tau \M:=T \M^{\otimes \sigma }\otimes T^* \M^{\otimes \tau }, 
$$ 
where $T \M $ and $T^* \M $ are the tangent and the cotangent bundle of $\M$, respectively.
Let $\mathcal{T}^\sigma_\tau \M$ denote the $C^\infty ( \M )$--module of all smooth 
sections of $T^\sigma_\tau \M$.

Throughout the rest of this paper, we will adopt the following convention.
\smallskip
\begin{mdframed}
\begin{itemize}
\item $\p$ always denotes a point on a uniformly regular manifold.
\item $k\in \N_0$ and $s\geq 0$.
\item $\sigma,\tau\in \N_0$, $V=V^{\sigma}_{\tau}:=\{T^{\sigma}_{\tau}\M, (\cdot|\cdot)_g\}$, $E=E^{\sigma}_{\tau}:=\{\R^{m^{\sigma}\times m^{\tau}},(\cdot|\cdot)\}$.
\end{itemize}
\end{mdframed}

\medskip

Setting $\Rk=\R^m$ for $\kappa\in \K$, we define 
$\boldsymbol{L}_{1,loc}(\R^m,E):=\prod_{\kappa}{L}_{1,loc}(\Rk,E),$ 
\begin{equation*}
\begin{aligned}
& \Rck:L_{1,loc}( \M ,V) \rightarrow L_{1,loc}( \Rk ,E),\quad u\mapsto \psi_{\kappa}^{\ast}({\pk}u), \\
&\Rek:L_{1,loc}( \Rk ,E) \rightarrow L_{1,loc}( \M ,V),\quad v_{\kappa} \mapsto  \pk \kb v_{\kappa}.
\end{aligned}
\end{equation*}
Here, and in the following, it is understood that a partially defined and compactly supported 
tensor field is automatically extended over the whole base manifold by identifying it to be zero
outside its original domain. We further introduce two maps:
\begin{equation*}
\begin{aligned}
&\Rc:  L_{1,loc}(\M,V) \rightarrow \boldsymbol{L}_{1,loc}(\R^m,E),
      &&u \mapsto (\Rck u)_{\kappa \in \K}, \\
&\Re: \boldsymbol{L}_{1,loc}( \R^m,E) \rightarrow L_{1,loc}(\M,V),
     &&(v_\kappa)_{\kappa \in \K} \mapsto \sum\limits_{\kappa\in \K} \Rek v_{\kappa}.
\end{aligned}
\end{equation*}

\subsection{H\"older and little H\"older spaces on uniformly regular manifolds}\label{Section 2.2}
In this subsection we follow Amann \cite{Ama12, Ama13}, see also \cite{ShaoSim14}.
We define  
$$
BC^k(\M,V):=(\{u\in{C^k(\M,V)}:\|u\|_{k,\infty}^{\M}<\infty\},\|\cdot\|_{k,\infty} ),
$$
where $\|u\|_{k,\infty} :={\max}_{0\leq i \leq k}\||\nabla^i u|_g\|_{\infty}$.
Set
$$
BC^\infty(\M,V):=\bigcap_k BC^k(\M,V)
$$
endowed with the conventional projective topology. Then
\begin{center}
$bc^k(\M,V):=$ the closure of $BC^{\infty}(\M,V)$ in $BC^k (\M,V)$.
\end{center}
Letting $k<s<k+1$, the  H\"older space  $BC^s(\M,V)$ is defined by
$$
BC^s(\M,V):=(bc^k(\M,V),bc^{k+1}(\M,V))_{s-k,\infty}.
$$
Here $(\cdot,\cdot)_{\theta,\infty}$ is the real interpolation method, see \cite[Example I.2.4.1]{Ama95}. 
For $s\geq 0$, we define the  little H\"older spaces by 
\begin{center}
$bc^s(\M,V):=$ the closure of $BC^\infty (\M,V)$ in $BC^s(\M,V)$. 
\end{center}
The  spaces  $BC^s(\R^m,E)$ and $bc^s(\R^m,E)$ are defined in a similar manner.
When $s\notin\N_0$, we can give an alternative characterization of these spaces on $\R^m$.
For $0<s<1$ and $0<\delta\leq\infty$, we define a seminorm by
$$
[u]^{\delta}_{s,\infty}:=\sup_{h\in(0,\delta)^m}\frac{\|u(\cdot+h)-u(\cdot)\|_{\infty}}{|h|^s}, \quad [\cdot]_{s,\infty}:=[\cdot]^{\infty}_{s,\infty}.
$$
For $k<s<k+1$,  the space  $BC^s(\R^m,E)$ can be equivalently defined as 
\begin{align*}
BC^s(\R^m,E)=\big(\{u\in BC^k(\R^m,E):\|u\|_{s,\infty}< \infty \},\|\cdot\|_{s,\infty}\big),
\end{align*}
where $\|u\|_{s,\infty}:=\|u\|_{k,\infty}+\max_{|\alpha|=k}[\partial^{\alpha} u]_{s-k,\infty}$; and
\begin{center}
$u\in BC^s(\R^m,E)$ belongs to $bc^s(\R^m,E)$ iff $\lim\limits_{\delta\rightarrow 0}[\partial^{\alpha}u]^{\delta}_{s-[s],\infty}=0,\quad |\alpha|=[s]$.
\end{center}
For $\F \in \{bc,BC\}$, we put $\boldsymbol{\F}^s:=\prod_{\kappa}{\F}^s_{\kappa}$ with ${\F}^s_{\kappa}:={\F}^s(\Rk,E)$. We denote by $l_{\infty}(\boldsymbol{\F}^s)$ the linear subspace of $\boldsymbol{\F}^s$ consisting of all $\boldsymbol{x}=(x_\kappa)_{\kappa \in \K}$ such that
$$
\|\boldsymbol{x}\|_{l_{\infty}(\boldsymbol{\F}^s)}:=\sup_{\kappa}\|x_{\kappa}\|_{\F^s_{\kappa}}<\infty.
$$
We define
$
l_{\infty,\uf}(\boldsymbol{bc}^k)
$
as the linear subspace of $l_{\infty}  (\boldsymbol{bc}^k)$ 
consisting of all $\boldsymbol{u}=(u_{\kappa})_{\kappa \in \K}$ 
such that $(\partial^{\alpha}u_\kappa)_{\kappa \in \K}$ is uniformly continuous on $\Rk$ for $|\alpha|\leq k$, uniformly with respect to $\kappa\in\K$. 
For  $k<s<k+1$, we define
$
l_{\infty,\uf}(\boldsymbol{bc}^s)
$
as the linear subspace of $l_{\infty,\uf}(\boldsymbol{bc}^k)$ of all $\boldsymbol{u}=(u_{\kappa})_{\kappa \in \K}$ such that 
\begin{align}
\label{S2: infnty,uf}
\lim\limits_{\delta\rightarrow 0}\max_{|\alpha|=k}[\partial^{\alpha}u_{\kappa}]^{\delta}_{s-k,\infty}=0 \quad \text{uniformly with respect to } \kappa\in\K.
\end{align}

The following properties of little H\"older spaces were first established in \cite{Ama13, Ama12}.
We also refer to {\cite[Theorem 2.1 and Proposition 2.2]{ShaoSim14}.

\begin{prop}
\label{S3: Prop: retraction of bc}
Let $s\geq 0$. Then $\Re$ is a retraction from $l_{\infty,\uf}(\boldsymbol{bc}^s)$ onto $bc^s(\M,V)$
with $\Rc$ as a coretraction. Similarly, 
$$
[u\mapsto (\psi_\kappa^* ( \zeta_\kappa u) )_{\kappa \in \K}	] \in \L( bc^s(\M,V), l_{\infty,\uf}(\boldsymbol{bc}^s)).
$$ 
\end{prop}

Let $(\cdot,\cdot)^0_{\theta,\infty}$ denote the continuous interpolation method, c.f. \cite[Example I.2.4.4]{Ama95}. 
\begin{prop}
\label{S3: Prop: interpolation of bc}
Suppose that $0<\theta<1$, $0\leq{s}_0<s_1$ and $s=(1-\theta)s_0+\theta{s_1}$ with $s_1,s_2,s\notin\N_0$. Then
$$
(bc^{s_0}({\M},V),bc^{s_1}({\M},V))_{\theta,\infty}^0\doteq bc^{s}({\M},V).
$$
\end{prop}

\subsection{Continuous maximal regularity}\label{Section 2.3}
For a fixed interval $I=[0,T]$,  $\mu\in(0,1)$, and a given Banach space $X$, we define
\begin{align*}
&BC_{1-\mu}(I,X):=\{u\in{C(\dot{I},X)}:[t\mapsto{t^{1-\mu}}u]\in C(\dot{I},X),\lim\limits_{t\to 0^+ } t^{1-\mu}\|u(t)\|_X=0\},\\
& \|u\|_{C_{1-\mu}}:=\sup_{t\in{\dot{I}}} t^{1-\mu} \|u(t)\|_{X},
\end{align*}
where $\dot{I}=I\setminus\{0\}$;
and
$$
BC_{1-\mu}^1(I,X):=\{u\in{C^1(\dot{I},X)}: u,\dot{u}\in BC_{1-\mu}(I,X)\}.
$$
If $I=[0,T)$ is a half open interval, then
\begin{align*}
&C_{1-\mu}(I,X):=\{v\in C(\dot{I},X) : v \in BC_{1-\mu}([0,t],X) ,\quad t<T\},\\
&C^1_{1-\mu}(I,X):=\{v\in C^1(\dot{I},X): v,\dot{v}\in C_{1-\mu}(I,X) \}.
\end{align*}
We equip these two spaces with the natural Fr\'echet topology induced by the topology of $BC_{1-\mu}([0,t],X)$ and $BC_{1-\mu}^1([0,t],X)$, respectively. 
\smallskip\\
Assume that $E_1 \overset{d}{\hookrightarrow} E_0$ is a pair of densely embedded Banach spaces. 
An operator $A$ is said to belong to the class $\cH(E_1,E_0)$, if $-A$ generates a strongly continuous analytic semigroup on $E_0$ with $dom(A)=E_1$. 
We define
\begin{align}
\label{S3: ez&ef}
{\ez}(I) := BC_{1-\mu}(I,E_0), \quad 
{\ef}(I) := BC_{1-\mu}(I,E_1)\cap BC_{1-\mu}^1(I,E_0) ,
\end{align}
which are themselves Banach spaces when equipped with the norms
\[
\begin{split}
\| v \|_{{\ez}(I)} &:= \sup_{t \in I} t^{1 - \mu} \|v(t)\|_{E_0}, \\
	\| v \|_{\ef(I)} &:= \sup_{t \in I} t^{1 - \mu} 
		\big( \|\dot{v}(t)\|_{E_0} +\|v(t)\|_{E_1} \big),
\end{split}
\]
respectively.
For $A \in\cH(E_1,E_0)$, we say $({\ez}(I),{\ef}(I))$ is a pair of maximal regularity of ${A}$  if
$$
\displaystyle \Big( \frac{d}{dt} + A, \gamma_0\Big)\in \Lis ({\ef}(I),{\ez}(I)\times E_\mu ) 
$$
where $\gamma_0$ is the evaluation map at $0$, i.e., $\gamma_{0}(u)=u(0)$, and $E_\mu:=(E_0,E_1)_{\mu,\infty}^0$. 
In this case, we use the notation
$$
A \in \mathcal{M}_\mu(E_1,E_0).
$$
\subsection{Quasilinear equations with singular nonlinearity}\label{Section 2.4}
Consider the following abstract quasilinear parabolic evolution equation
\begin{equation}
\label{S2: Abstract EE}
\left\{\begin{aligned}
\frac{d}{dt} u +A(u) u &= F_1(u) +F_2(u) , \quad t>0, \\
u(0)&=x.
\end{aligned}\right.
\end{equation}
We assume that $V_\mu\subset E_\mu$ is an open subset of the 
continuous interpolation space $E_\mu:= (E_0,E_1)_{\mu,\infty}^0$ and 
the operators $(A, F_1, F_2)$ satisfy
the following conditions.
\begin{itemize}
\item[\textbf{(H1)}] Local Lipschitz continuity of $(A, F_1)$:
$$
(A,F_1)\in C^{1-}(V_\mu, \mathcal{M}_\mu(E_1,E_0)\times E_0).
$$
\item[\textbf{(H2)}] Structural regularity of $F_2$:

\noindent There exists a number $\gamma\in (\mu,1)$ such that $F_2: V_\mu \cap E_\gamma \to E_0$. \
Moreover, there are numbers $\gamma_j \in [\mu,\gamma]$, $\varrho_j\geq 0$, and $m\in\N$ with
\begin{equation}
\label{S2: structure reg 1}
\frac{\varrho_j (\gamma-\mu)+ (\gamma_j-\mu)}{1-\mu}\leq 1,\quad \text{for all }j=1,2,\cdots,m,
\end{equation} 
so that for each $x_0\in V_\mu$ and $R>0$ there is a constant $C_R=C_R(x_0)>0$ for which the estimate
\begin{equation}
\label{S2: structure reg 2}
|F_2(x_1)-F_2(x_2)|_{E_0}  \leq C_R \sum\limits_{j=1}^m (1 + |x_1|^{\varrho_j}_{E_\gamma} +  |x_2|^{\varrho_j}_{E_\gamma} )|x_1-x_2|_{E_{\gamma_j}}
\end{equation} 
holds for all $x_1,x_2\in \bar\B_{E_\mu}(x_0,R)\cap (V_\mu\cap E_\gamma)$.
\end{itemize}

Following the convention in  \cite{PruWil17} and \cite{LeCroneSim18}, we call the index
$j$ {\em subcritical} if \eqref{S2: structure reg 2} is a strict inequality and {\em critical}
in case equality holds in \eqref{S2: structure reg 2}.

\begin{theorem}\label{S2: Thm MR}
\cite[Theorem~2.2]{LeCroneSim18}
Suppose $(A, F_1,F_2)$ satisfies  {\bf (H1)--(H2)}. 
\begin{itemize}
\item[(a)] Given any $x_0\in V_\mu$, there exist positive constants $\tau=\tau(x_0),$ $\varepsilon=\varepsilon(x_0),$ and $\sigma=\sigma(x_0)$ such that \eqref{S2: Abstract EE} has a unique solution
$$
u(\cdot,x)\in \ef([0,\tau])
$$
for all initial values $x \in \bar{\B}_{E_\mu}(x_0,\varepsilon)$. Moreover,
$$
\|u(\cdot,x_1) - u(\cdot,x_2)\|_{\ef([0,\tau])} \leq \sigma \|x_1-x_2\|_{E_\mu},\quad x_1,x_2\in \bar{\mathbb{B}}_{E_\mu}(x_0,\varepsilon).
$$
\item[(b)] Each solution with initial value $x_0 \in V_\mu$ exists on a
maximal interval $J(x_0):=[0,t^+)=[0,t^+(x_0))$ and enjoys the regularity
$$
u(\cdot,x_0)\in C([0,t^+),E_\mu)\cap C((0,t^+),E_1).
$$
\item[(c)] If the solution $u(\cdot,  x_0)$ satisfies the conditions:
\begin{itemize}
\item[(i)] $u(\cdot,  x_0) \in UC(J(x_0), E_\mu)$ and
\item[(ii)] there exists $\eta>0$ so that ${\rm dist}_{E_\mu}(u( t ,  x_0),\partial V_\mu)>\eta$ for all $t\in J(x_0)$,
\end{itemize}
then it holds that $t^+(x_0)=\infty$ and so $u(\cdot,  x_0)$ is a global solution of \eqref{S2: Abstract EE}
Moreover, if the embedding $E_1 \hookrightarrow E_0$ is compact, then condition (i) may be
replaced by the assumption:
\begin{itemize}
\item[(i.a)] the orbit $\{u(t,x_0): t\in [\tau,t^+(x_0))\}$ is bounded in $E_\delta$ for some $\delta\in (\mu,1]$ and some $\tau \in (0, t^+(x_0))$.
\end{itemize}
\end{itemize}
\end{theorem}

\section{URT--hypersurfaces}\label{UR}
Suppose $\Sigma$ is an oriented smooth hypersurface without boundary which is embedded in $\R^{m+1}$. 
Let $\a>0$. 
Then $\Sigma$ is said to have a {\em tubular neighborhood of radius $\a$}
if the map
\begin{equation}
\label{S3: Diffeomorphism X}
X:  \Sigma \times (-\a,\a) \to \R^{m+1}: \;\; [(\p,r) \mapsto \p + r \nu_{\Sigma}(\p)]
\end{equation}
is a diffeomorphism onto its image $U_\a:=X((-\a,\a)\times \Sigma)$. 
Here $\nu_\Sigma$  is the normal unit vector field 
compatible with the orientation of $\Sigma$. 
We refer to $U_\a$ as the tubular neighborhood of $\Sigma$ of width  $2\a$ and 
note that $U_\a=\{x\in\R^{m+1}: {\rm dist}\,(x,\Sigma)<\a\}$.

Finally,  we say that $\Sigma$ has {\em a tubular neighborhood} if 
there exists a number $\a>0$ such that the above property holds.

\begin{rmk}\label{S3: Rmk: tubular-nbgh}
(a) We lose no generality in assuming $\Sigma$ is oriented, as any smooth 
embedded hypersurface without boundary is orientable, cf. \cite{Sam68}. 

\smallskip\noindent
(b) Any smooth (in fact, $C^2$) compact embedded hypersurface  without boundary has a tubular neighborhood,
see for instance \cite[Exercise 2.11]{GilTru01}. 

\smallskip\noindent
(c) Suppose $\Sigma$  is a smooth (oriented) embedded hypersurface with unit normal field $\nu_\Sigma$.
Then $\Sigma$ is said to satisfy the {\em uniform ball condition of radius $\a>0$} if
at each point  $\p\in\Sigma$, the open balls $\B(\p \pm \a \nu_\Sigma (\p) , \a)$ do not intersect $\Sigma$.

The following assertions are equivalent:
\begin{itemize}
\item[(i)]  $\Sigma$ has a tubular neighborhood of radius $\a$.
\item[(ii)] $\Sigma$ satisfies the uniform ball condition of radius $\a$.
\end{itemize}
For the reader's convenience, we include a proof of this equivalence.
\begin{proof}
{(i) $\Rightarrow$ (ii).}
Suppose there exists $\p\in\Sigma$ such that $\B(x_0,\a)\cap \Sigma \neq \emptyset,$ where $x_0:=\p+\a\nu_\Sigma (\p)$.
Then $s:={\rm dist}\,(x_0,  \bar\B (x_0,\a)\cap \Sigma )<\a$ and there exists $\q\in \bar\B (x_0,\a)\cap \Sigma$
such that $x_0=\q + s \nu_\Sigma(\sf q)$. Hence
$x_0=X(\p,\a)= X(\q,s)$, with $(\p, \a)\neq (\q,s)$, contradicting the assumption that $X$ is bijective.
The case $x_0=\p -\a\nu_\Sigma  (\p) $ is treated in the same way.

\smallskip\noindent
{(ii) $\Rightarrow$ (i).}
We only need to prove the injectivity of $X$. 
Suppose, by contradiction, that $X(\p_1,r_1)=X(\p_2,r_2)=x$ for $(\p_1,r_1)\neq (\p_2,r_2)$.
Without loss of generality  we may assume that $r_1 \in (0,\a)$, as we can otherwise 
replace $\nu_\Sigma(\p_i)$ by $- \nu_\Sigma(\p_i)$. Moreover, we may assume
that $|r_2|\le r_1$. Let $s\in (r_1,\a)$ and set 
$$y:=X(\p_1,s)=x+(s-r_1)\nu_\Sigma  (\p_1)  .$$
Then we have 
$|y-\p_2| \le |y-x| + |x-\p_2| = s-r_1 + |r_2| \le s,$
showing that
$$\p_2\in \bar\B (y,s)=\bar \B(\p_1 + s\nu_\Sigma (\p_1),s)\subset \B(\p_1 + \a \nu_\Sigma (\p_1),\a).$$
Therefore,  
$\B(\p_1 + \a\nu_\Sigma (\p_1),\a)\cap \Sigma \neq\emptyset,$ 
contradicting the assumption in (ii). 
\end{proof}
\noindent 
(d) Suppose $\Sigma$ has a tubular neighborhood of radius $\a$. Let $\{\kappa_1,\ldots,\kappa_m\}$
be the principal curvatures of $\Sigma$, and $L_\Sigma$ the Weingarten tensor.

Then it follows from part   (c)  that
$|\kappa_1|,\ldots, |\kappa_m|\le1/ \a$ and   $|L_\Sigma|\le 1/\a.$
\end{rmk}
\goodbreak 
\noindent
In the following, we say that $\Sigma$ is a (URT)--hypersurface in $\R^{m+1}$ if
\begin{itemize}
\item[(T1)] $\Sigma$ is a smooth oriented hypersurface without boundary embedded in $\R^{m+1}$.
\item[(T2)] $(\Sigma,g)$ is uniformly regular, where $g=g_{m+1}|_{\Sigma}$ denotes the metric induced by the Euclidean metric $g_{m+1}$.
\item[(T3)] $\Sigma$  has a tubular neighborhood.
\end{itemize}

\begin{examples}
\label{ex:URT}
(a) Every smooth compact hypersurface without boundary embedded in $\R^{m+1}$ 
is a (URT)--hypersurface.

\smallskip\noindent
(b) All of the manifolds considered in \cite{LeCroneSim13, LeCroneSim16} are (URT)--hypersurfaces. 
In particular, the infinite cylinder with radius $r > 0$, 
$$
\mathcal{C}_r= \{(x,y,z)\in \R^3: y^2+z^2=r^2, \; x\in \R\},
$$
is a (URT)--hypersurface with tubular neighborhood of radius $\a = r$.

\smallskip\noindent
(c) Assume that $f: \R^m \to \R$ belongs to $BC^2(\R^m)$. Then the graph of $f$ has a tubular neighborhood of radius $\a$ for some $\a>0$.
\begin{proof}
By the inverse function theorem, there exist uniform constants $\eta>0$ and $\varepsilon>0$ such that, 
at every point $x\in \R^m$, $f|_{ \Bm(x,\eta)}$ can be expressed as the graph of a $BC^2$--function
$h_x$  over  $T_x {\rm gr}(f)$, the tangent space to the graph of $f$ at the point $(x,f(x))$,
such that the set 
$ \{(y, h_x(y)): y\in \B_{T_x {\rm gr}(f)}(0,\varepsilon)\}$ is contained in $\{(z,f(z)): z\in \B^m(x,\eta)\}$. 
Moreover,  there exists a uniform constant $c,$ independent of $x$, such that
\begin{equation}
\label{zeta_x bound}
\|h_x\|_{2,\infty} \leq c
\end{equation}
where the supremum is taken over the ball  $\B_{T_x {\rm gr}(f)}(0,\varepsilon)$.
We  refer to the proof of Claim 1 in Proposition~\ref{A-Gamma_rho URT}(b) in the Appendix for a more general situation.
Further, we have $h_x(0)=0$ and $\nabla h_x(0)=0$. 
Due to \eqref{zeta_x bound}, after 
Taylor expansion of $h_x$ around $0\in T_x {\rm gr}(f)$, we have
$$
|h_x(y)| \leq  \|\nabla_y^2 h_x \|_\infty  |y|^2,\quad y\in  \B_{T_x {\rm gr}(f)}(0,\varepsilon),
$$
for sufficiently small $\varepsilon$. Choosing $C \ge \|\nabla_y^2 h_x \|_\infty$ such that
$1/2C \leq \varepsilon,$ we define $\a := 1/2C$.  It follows that the ball
$\B^{m+1}(\a\nu_x,\a)$
lies above the graph 
$$\{(y,h_x(y)): y\in T_x {\rm gr}(f)(0,\varepsilon)\},$$ 
where $\nu_x$ is the upwards pointing unit normal of ${\rm gr}(f)$ at the point $(x,f(x))$.
An analogous argument shows that the ball $\B^{m+1}(-\a\nu_x,\a)$ lies below the graph. 

Since the constants $\varepsilon$ and $\a$ are independent of $x$, combining with Remark 3.1(c), 
this proves that ${\rm gr}(f)$ has a tubular neighborhood of radius $\a$.
\end{proof}
\goodbreak
\smallskip\noindent
(d) We refer to \cite{Ama13, Ama12, Ama15} for additional examples of uniformly regular manifolds.
In particular, embedded hypersurfaces with tame ends, considered in  \cite[Theorem 1.2]{Ama15},
are  (URT)--hypersurfaces. More precisely, 
given a compact hypersurface without boundary $B$, embedded in $\R^m$,
and  $0\leq \alpha\leq 1,$ we define 
$$
F_\alpha(B):=\{ (t,t^\alpha y): t>1, y\in B \},
$$
which we endow with the metric $g_{F_\alpha(B)}$
induced by its embedding into $\R^{m+1}$.
An embedded hypersurface $\Sigma \subset \R^{m+1}$ is said to have tame ends if
$$
\Sigma=V_0 \cup \bigcup_{i=1}^n V_i,
$$
where $(V_0, g_{m+1}|_{V_0})$ is compact and  $(V_i, g_{m+1}|_{V_i})$
is isometric to  $(F_\alpha(B), g_{F_\alpha(B)})$.
Then, $(\Sigma, g_{m+1}|_{\Sigma})$ is a (URT)--hypersurface.

In particular, when $\alpha=0$, $(\Sigma, g_{m+1}|_{\Sigma})$ has finitely many cylinder ends; 
when $\alpha=1$, $(\Sigma, g_{m+1}|_{\Sigma})$ 
has finitely many (blunt) cone ends. 

\smallskip\noindent
(e) Let
$$
\mathcal{C}_{k} = \{(x,y,z)\in \R^3: y^2+z^2=1+1/k,\; x\in \R\},\quad k\in\N.
$$
Based on part (b), the manifold $\Sigma= \bigcup_k \mathcal{C}_{k},$ endowed with the metric 
induced by $g_3$, is uniformly regular. 
But it is obvious that $(\Sigma,g)$ does not have a tubular neighborhood.

\smallskip\noindent
(f) There also exist connected uniformly regular hypersurfaces that are not (URT). 
For instance, 
we can construct a smooth connected curve $C$ in $\{(x,y): y>0\}$ such that 
$C\cap \{(x,y):  x\geq 0\}$ is compact and
$$
C\cap \{(x,y): x< 0\}= \{(x,y): y=1\}\cup  \{(x,y): y=1+e^x\}.
$$
Then $(C,g_2|_{C})$ is a uniformly regular hypersurface that is not (URT). 
One can take the product of $C$ with $\R^m$ to produce higher dimensional examples.

\medskip\noindent
Additionally, one can rotate the curve $C$ around the $x$--axis to obtain a connected 
rotationally symmetric uniformly regular hypersurface which is not (URT).

\end{examples}

\section{The surface diffusion flow}\label{Section 4}

In solving the surface diffusion flow, one seeks to find a family of (oriented) closed  hypersurfaces  
$\{\Gamma(t): t\geq 0\}$ satisfying the evolution equation
\begin{equation}
\label{S1: SDF}
\left\{\begin{aligned}
V(t)&= -\Delta_{\Gamma(t)}H_{\Gamma(t)} ,  \quad t > 0,\\
\Gamma(0)&=\Gamma_0,
\end{aligned}\right.
\end{equation}
for an initial hypersurface $\Gamma_0$.

Here, $V(t)$ denotes the velocity in the normal direction of $\Gamma$ at time $t$, $H_{\Gamma(t)}$ is 
the mean curvature of $\Gamma(t)$ (i.e., the average of the principal curvatures),
and $\Delta_{\Gamma(t)}$ is the Laplace-Beltrami operator on $\Gamma(t)$.  
We use the convention that  a sphere has {\em negative} mean curvature.
We note that this convention is in agreement with \cite{PruSim16, Shao13, Shao15},
but differs from  \cite{EscMaySim98, LeCroneSim13, LeCroneSim18}.

\medskip
In the following, we assume that  $\Sigma$ is a  (URT)--hypersurface in $\R^{m+1}$ with 
tubular neighborhood $U_\a$ and with an orientation-preserving atlas $\A:=\{( \Uk ,\varphi_\kappa): \kappa\in \K\}$ 
with $\psi_\kappa= \varphi_\kappa^{-1}$ satisfying (R1)--(R5). 
In the following,  we assume that $\Sigma$ carries the metric induced by the Euclidean 
metric $g_{m+1}$.
Finally, we assume that $\Gamma_0$ lies in $U_\a$.

For $\alpha\in (0,1)$ a fixed parameter, we define
$$
E_0:=bc^\alpha(\Sigma) \quad \text{and} \quad E_1:=bc^{4+\alpha}(\Sigma).
$$
For $\theta\in (0,1)$, let $E_{\theta}:=(E_0,E_1)_{\theta,\infty}^0$.
Taking $\mu=1/4$ and $\gamma=3/4$, it follows from Proposition~\ref{S3: Prop: interpolation of bc} that
$$
E_\mu=bc^{1+\alpha}(\Sigma) \quad \text{and} \quad E_\gamma=bc^{3+\alpha}(\Sigma).
$$

Given $\rho \in E_\mu$ with $\| \rho \|_{\infty} < \a$, 
it follows, by assumption that $\Sigma$ is (URT) with tubular
neighborhood $U_\a$, that
\begin{equation} \label{psi-rho}
\Psi_\rho:   \Sigma \to \R^{m+1}, \quad \Psi_\rho( \p)=\p + \rho( \p) \nu_{\Sigma}(\p),
\end{equation}
is a diffeomorphism from $\Sigma$ onto the $C^1$--manifold
$\Gamma_\rho := {\rm im}(\Psi_\rho)$;
see also Proposition~\ref{A-Gamma_rho URT} for additional properties of $\Gamma_\rho$.

When the temporal variable $t$ is included in $\rho$, i.e. 
$$
\rho:[0,T) \times \Sigma \to (-\a,\a),
$$ 
we can also extend $\Psi_\rho$ to $\Psi_\rho: [0,T)\times  \Sigma \to \R^{m+1}$.
In the sequel, we will omit the temporal variable $t$ in $\rho$, $\Psi_\rho$ and $\Gamma_\rho$ when the dependence on $t$ is clear from context. 

\medskip\noindent
Let us fix some notation. 
We denote by  $g_{m+1}|_{\Gamma_\rho}$ the metric induced on $\Gamma_\rho$ by the 
Euclidean metric $g_{m+1}$ of $\R^{m+1}$. 
Let $g(\rho):= \Psi_\rho^*\, (g_{m+1}|_{\Gamma_\rho})$ be the pull-back metric of $g_{m+1}|_{\Gamma_\rho}$ on $\Sigma$.

The following expression for $g(\rho)$ was derived in \cite[Formula~(23)]{PruSim13}:
\begin{equation}
\label{S3: gamma_ij}
g_{ij}(\rho)=g_{ij}-2\rho{l}_{ij}+\rho^{2}l^{r}_{i}l_{jr}+\partial_{i}\rho\partial_{j}\rho,
\end{equation}
where $l^i_j$ and $l_{ij}$ are the components of the Weingarten tensor $L_{\Sigma}$ and the second
fundamental form with respect to $g:=g_{m+1}|_\Sigma$; i.e.,
$$
 l_{ij}=-( \tau_i |\partial_j \nu_\Sigma),\quad  l^i_j= g^{ik}l_{kj}, \quad
L_\Sigma = l_{ij} \tau^i\otimes \tau^j  = l^i_j  \tau_i\otimes \tau^j,
$$
where 
$
\{\tau_1,\ldots,\tau_m\}=\{\frac{\partial} {\partial x^1},\ldots, \frac{\partial} {\partial x^m}\}
$
is a local basis of $T\Sigma$ at $\p$ and 
$\{\tau^1,\ldots,\tau^m\}=\{dx^1,\ldots, dx^m\}$ 
is the corresponding dual basis,
characterized by $(\tau^i | \tau_j)=\delta^i_j$.

We introduce an open subset of $E_\mu$ defined by
$$
V_\mu := \{\rho\in E_\mu: \, \|\rho\|_\infty <\a\}.
$$
By Remark~\ref{S3: Rmk: tubular-nbgh}(d),
the functions 
\begin{equation} \label{a(rho)-beta(rho)}
a(\rho):=(I-\rho L_{\Sigma} )^{-1}  \nabla_\Sigma \rho,\quad \beta(\rho):=[1+ |a(\rho)|^2]^{-1/2},
\end{equation}
are well--defined for all $\rho\in V_\mu$, where $\nabla_\Sigma \rho$ is the gradient vector
and $I=\tau_i\otimes \tau^i$.

It is easy to verify that
$g_{ij}(\rho)=(\tau_i| K(\rho)\tau_j)$, where
$$
K(\rho)= (I-\rho L_\Sigma)^2 + \nabla_\Sigma \rho \otimes \nabla_\Sigma \rho 
           =  (I -\rho L_\Sigma)[I + a(\rho) \otimes a(\rho)](I- \rho L_\Sigma).
$$ 
Hence, we obtain
\begin{equation}\label{g(rho)-ij}
g_{ij}(\rho)=\Big( (I -\rho L_\Sigma)\tau_i \Big| [I + a(\rho) \otimes a(\rho)](I- \rho L_\Sigma)\tau_j\Big).
\end{equation}
It follows from the well-known relation
$$[I +a\otimes a]^{-1}= I-\frac{a\otimes a}{1+|a|^2},\quad a\in\R^{m+1},$$
that $K(\rho)$ is invertible for every $\rho\in C^1(\Sigma)$ with $\|\rho\|_\infty < \a$, 
with inverse given by
$$
K^{-1}(\rho) = M_0(\rho)[I - \beta^2(\rho) a(\rho)\otimes a(\rho) ] M_0(\rho),
$$
where $M_0(\rho):= (I-\rho L_\Sigma)^{-1}$. 
We then have $g^{ij}(\rho)=(\tau^i | K^{-1}(\rho) \tau^j)$
for the components of the cotangent metric $g^*(\rho)$ on $T^*\Sigma$ induced by $g(\rho)$,
and hence
\begin{equation}\label{g(rho)-inverse-ij}
g^{ij}(\rho)=\Big( M_0(\rho) \tau^i \Big| [I - \beta^2(\rho) a(\rho)\otimes a(\rho) ] M_0(\rho) \tau^j \Big),
\end{equation}
see also \cite[Section 2.2]{PruSim16}.

When parameterizing the evolving hypersurface $\Gamma(t)=\Gamma_{\rho(t)}$ by means of a 
height function $\rho(t) \in V_1 = V_\mu \cap E_1,$ it holds that \eqref{S1: SDF} is equivalent to 
\begin{equation}
\label{S3: rho expression}
\partial_t \rho = - \frac{1}{\beta(\rho)} \Psi_\rho^* ( \Delta_{g_{m+1}|_{\Gamma_\rho}} H_{\Gamma_\rho})= - \frac{1}{\beta(\rho)}  \Delta_\rho H_\rho .
\end{equation}
Here, $\Delta_{g_{m+1}|_{\Gamma_\rho}}$ and $\Delta_\rho$ denote Laplace-Beltrami 
operators on $(\Gamma_\rho, g_{m+1}|_{\Gamma_\rho})$ and $(\Sigma, g(\rho))$, respectively. 
It was shown in \cite[Section~5]{Shao15} that
 $H_\rho:=\Psi_\rho^* H_{\Gamma_\rho}$ in each local patch $( \Uk ,\varphi_{\kappa})$ reads as
\begin{align} \label{S3: H_rho expression}
H_\rho &= \frac{\beta(\rho)}{m}\Big\{ g^{ij}(\rho) \partial_i\partial_j\rho
+ g^{ij}(\rho) ( l^k_j \partial_i \rho-\Gamma^k_{ij})\partial_k \rho   \\
& + g^{ij}(\rho) \big[r_k^l (\rho)l^k_i \partial_j \rho +
r_k^l(\rho)(\partial_j l_i^k + \Gamma^k_{jh} l_i^h -\Gamma^h_{ij} l_h^k)\rho 
+r_k^l(\rho) l^h_j l^k_h \rho\partial_i \rho \big]\Big\} \partial_l \rho \notag\\
& + \frac{\beta(\rho)}{ m}  g^{ij}(\rho) ( l_{ij} - l_{ik} l^k_j \rho), \notag
\end{align}
with $g^{ij}(\rho)$ given in \eqref{g(rho)-inverse-ij}. Here, $\Gamma^k_{ij}$ are the components of the Christoffel 
symbols of $\Sigma$ associated with the metric $g=g_{m+1}|_{\Sigma}$, and
$r^i_j(\rho)={p^i_j(\rho)}/{q^i_j(\rho)},$
where $p^i_j(\rho)$ and $q^i_j(\rho)$ are polynomials of $\rho$ with $BC^\infty$--coefficients.
Note that although \eqref{S3: H_rho expression} was derived for compact hypersurfaces in 
\cite{Shao15}, this expression still holds true for our problem as it is purely local. 

In local coordinates with respect to the atlas $\A$,  $\Delta_\rho$ is given by
\begin{equation}
\label{Beltrami-rho}
\Delta_\rho = g^{ij}(\rho)(\partial_i\partial_j -\Gamma^k_{ij}(\rho)\partial_k\;),
\end{equation}
where $\Gamma^k_{ij}(\rho)$ are the Christoffel symbols of $(\Sigma, g(\rho))$.
Here we note that the terms $\Gamma^k_{ij}(\rho)$ depend on $\rho$ and 
up to its second--order derivatives. 
 More precisely, 
$$
\displaystyle \Gamma^k_{ij}(\rho)=\frac{p^k_{ij}(\rho, \partial \rho, \partial^2 \rho)}{q^k_{ij}(\rho, \partial \rho)},
$$ 
where $p^k_{ij}$ is a polynomial of $\rho$ and its derivatives up to second order and 
$q^k_{ij}$ is a polynomial of $\rho$ and its first--order derivatives
(both polynomials having $BC^\infty$--coefficients).

By the expression above, we obtain
$$H_\rho=\frac{\beta(\rho)}{m} \ev( g^*(\rho), \nabla^2 \rho) + \text{lower order terms}$$ 
and
$$
\Delta_\rho H_\rho =\frac{1}{m} \ev( g^*(\rho)\otimes g^*(\rho), \nabla^4 \rho) + \text{lower order terms},
$$
where $\ev(\cdot,\cdot)$ denotes the complete contraction and $\nabla$ is the covariant derivative with respect to
$(\Sigma, g)$. Here and in the sequel, we will still use $\nabla$ to denote its extension to $\mathcal{T}^\sigma_\tau\Sigma$.
Note that for $u\in C^4(\Sigma)$, the tensor $\nabla^4u\in C(\Sigma,T^*\Sigma^{\otimes 4})$ can be expressed in local coordinates by
\begin{equation*}\label{nabla-4-local}
\nabla^4u = \partial_{(j)}u\; \tau^{(j)} +\sum_{\beta, (j)}  a_{\beta, (j)}\partial^\beta u \,\tau^{(j)},
\end{equation*}
with coefficients $a_{\beta, (j)}\in BC^\infty(\B^m)$, where the summation runs over all multi-indices 
$(j)=(j_1,\cdots,j_4)\in \{1,\cdots,m\}^4$ and all 
$\beta \in \N^m$ with $|\beta| \le 3$. Here we are also using
$$
\partial_{(j)}=\partial_{j_1} \partial_{j_2}\partial_{j_3}\partial_{j_4},\quad 
\tau ^{(j)}=\tau^{j_1}\otimes \tau^{j_2}\otimes \tau^{j_3}\otimes \tau^{j_4},
$$
and $\partial^{\beta} := \partial^{\beta_1}_{1}\cdots\partial^{\beta_m}_{m}$;
see for instance \cite[page 444]{Ama13}.
Hence we obtain
\begin{equation}
\label{C(g,nabla)-local}
\ev( g^*(\rho)\otimes g^*(\rho), \nabla^4 u)=
g^{ij}(\rho)g^{lm}(\rho)\partial_i\partial_j\partial_l\partial_m u 
+\sum_{ 0<|\beta|\le 3} b_\beta(\rho,\partial \rho) \partial^\beta u
\end{equation}
for each  $\rho \in BC^1(\Sigma)$ with $\|\rho\|_{\infty}<\a$ and $u\in C^4(\Sigma)$.

By defining
\begin{equation}
\label{S4: Ar Fr}
A(\rho) \rho := \frac{1}{m} \ev( g^*(\rho)\otimes  g^*(\rho), \nabla^4 \rho),\quad  F(\rho) :=  A(\rho)\rho - \frac{1}{\beta(\rho)} \Delta_\rho H_\rho,
\end{equation}
we obtain an equivalent formulation of \eqref{S1: SDF} as
\begin{equation}
\label{S3: SDF-rho}
\left\{\begin{aligned}
\partial_t \rho +A(\rho)\rho &= F(\rho)  &&\text{in}&& (0,\infty)\times \Sigma, \\
\rho(0) &=\rho_0  &&\text{in }&&\Sigma .&&
\end{aligned}\right.
\end{equation}
We note that for each $\rho\in C^1(\Sigma,\R)$ with $\|\rho\|_\infty<\a$, the mapping
$$
A(\rho): C^4(\Sigma,\R)\to C(\Sigma,\R):\quad [u\mapsto  \frac{1}{m} \ev( g^*(\rho)\otimes  g^*(\rho), \nabla^4 u)]
 $$
gives rise to  a differential operator of order 4.

\smallskip\noindent
A linear operator 
$$\cA := \sum\limits_{i=0}^l \ev(a_i, \nabla^i \;\cdot\;),\quad u\mapsto  \cA u =\sum\limits_{i=0}^l \ev(a_i, \nabla^i  u), $$
of order $l$, acting on scalar functions, is said to be \emph{uniformly strongly elliptic}  if there exist
positive constants $r,R>0$ such that the principal symbol of $\cA$, 
\begin{align*}
\hat{\sigma}\mathcal{A}^{\pi}(\p,\xi):=\ev(a_l,(-i\xi)^{\otimes l})(\p)\in \R,\quad (\p,\xi)\in \Sigma\times T_p^*\Sigma,
\end{align*}
satisfies
\begin{equation}
\label{strongly-elliptic}
r\le {\rm Re}\,  \hat{\sigma}\mathcal{A}^{\pi}(\p,\xi)\le R,
\quad \text{for all  $(\p,\xi)\in \Sigma \times  T^\ast_{\p} \Sigma$ with $|\xi|_{g^{\ast}(\p)} = 1$}.
\end{equation}
\begin{remark}
\label{normally-elliptic}
In the scalar case, it is not difficult to see that the notion of uniformly strongly elliptic is equivalent to the notion of
{uniformly normally elliptic} introduced in \cite[Section 3]{ShaoSim14}, 
see also \cite{Ama17}.
\end{remark}

In our setting, the principal symbol of  $A(\rho)$ is given by
$$
\hat{\sigma}A^\pi(\rho)(\p,\xi)= |\xi|^4_{g^*(\rho)(\p)},\quad \xi\in T_\p^* \Gamma.
$$
It follows from \eqref{g(rho)-inverse-ij} that $g^*(\rho) \sim g^*$ 
for all $\rho\in V_\mu$,  
in the sense that there exists some $c\geq 1$ such that
\begin{center}
$(1/c) |\xi |_{g^*(\p)}^2 \leq |\xi |_{g^*(\rho)(\p)}^2 \leq c|\xi|^2_{g^*(\p)} \quad$ for  any $(\p,\xi)\in \Sigma \times T^*_{\p} \Sigma$.
\end{center}
In fact, note that with $\xi=\xi_i\tau^i \in T^*_{\p}\Sigma$,  \eqref{g(rho)-inverse-ij} implies
\begin{equation*}
|\xi|^2_{g^*(\rho)(\p)}=g^{ij}(\rho)(\p)(\xi,\xi)
= |M_0(\rho)\xi| ^2(\p) -\beta^2(\rho)(a(\rho)|M_0(\rho)\xi)^2(\p).
\end{equation*}
Next, observe that
\begin{equation}\label{equivalence of metrics}
\begin{aligned}
\beta^2(\rho)|M_0(\rho)\xi|^2(\p)\le g^{ij}(\rho)(\p)(\xi,\xi)\le |M_0(\rho)\xi|^2(\p),
\end{aligned}
\end{equation}
where we employed the Cauchy-Schwarz inequality and 
$1-\beta^2(\rho)|a(\rho)|^2=\beta^2(\rho)$ for the first estimate.
It remains to observe that
$$\min\Big\{\frac{1}{(1\!- \!\rho\kappa_r(\p))^2}\Big\} g^{ij}(\p)(\xi,\xi)
\le |M_0(\rho)\xi|^2(\p)
 \le \max\Big\{\frac{1}{(1\!-\!\rho\kappa_r(\p))^2}\Big\} g^{ij}(\p)(\xi,\xi),
 $$
where $\kappa_r$ are the principal curvatures of $\Sigma$, which are bounded by $1/\a$ since
$\Sigma$ satisfies a uniform ball condition of radius $\a$. 
This shows that  $A(\rho)$ is uniformly strongly elliptic.
Remark~\ref{normally-elliptic} and \cite[Proposition~2.7, Theorem~3.7]{ShaoSim14} now imply the following result.
\begin{prop}
\label{S3: Prop P(rho) MR}
$A \in C^\omega(V_\mu, \mathcal{M}_\mu (E_1,E_0)).$
\end{prop}

Next, we will verify that the operator $F$ satisfies \textbf{(H2)}. 
In each patch $(\Uk, \varphi_\kappa)$, 
we reference \cite[Section~4.4]{LeCroneSim18} and  \eqref{C(g,nabla)-local} to confirm that
the local expression for $F(\rho)$ is of the form
\begin{equation}
\label{S4: F structure}
F(\rho) =  \sum\limits_{|\eta|=3,|\tau|\leq 2} 
	c_{\eta,\tau}(\rho,\partial \rho) \, \partial^\tau \rho \, \partial^\eta \rho 
 	+ \sum\limits_{|\eta|,|\sigma|,|\tau|\leq 2} d_{\eta,\sigma,\tau}(\rho,\partial \rho) \,
 		 \partial^\eta \rho \, \partial^\sigma \rho \, \partial^\tau \rho,
\end{equation}
where $\eta, \tau, \sigma \in \mathbb{N}^m$ are multi--indices of length
$|\eta| := \eta_1 + \cdots + \eta_m$ and $\partial^{\eta} := \partial^{\eta_1}_{x_1}\cdots\partial^{\eta_m}_{x_m}$
is the mixed partial derivative operator in local coordinates.
The coefficient functions $c_{\eta,\tau}$ and $d_{\eta,\tau,\sigma}$ depend analytically on 
$\rho$ and its first--order derivatives. 
In the sequel, for a function $u:\Sigma\to \R$, we define $u_\kappa:= \zeta \kf u$.
Let $\rho_0\in V_\mu$. For $R>0$, we choose 
$\rho_1, \rho_2 \in \bar{\B}_{E_\mu}(\rho_0,R)\cap (V_\mu\cap E_\gamma)$.
By Proposition~\ref{S3: Prop: retraction of bc}, we have
\begin{align}
\notag 
\|F(\rho_1) -F(\rho_2)\|_{\alpha,\infty} &\leq C \|\Rc F(\rho_1) - \Rc F(\rho_2) \|_{l_\infty(\boldsymbol{BC}^\alpha)}\\
\label{S3: SDF est for Q}
&= C \sup\limits_{\kappa\in \K} \|\Rck F(\rho_1)  - \Rck F(\rho_1) \|_{\alpha,\infty}.
\end{align}
In the following computations,  $\tilde{C}$ denotes a generic constant depending only on $R$ 
and $\|\rho_0\|_{1+\alpha,\infty}$.
In every patch $(\Uk, \varphi_\kappa)$, by the discussion in \cite[Section~4.1]{LeCroneSim18}, 
we have the following estimate. 
\begin{align}
\notag \|\Rck &F(\rho_1)  - \Rck F(\rho_1) \|_{\alpha,\infty}\\
\label{S3: SDF est 1}
\leq & \, \tilde{C} \, \| \rho_{1,\kappa} - \rho_{2,\kappa}\|_{3+\alpha,\infty}\\
\notag + & \,\tilde{C} \Big( \|\rho_{1,\kappa} \|_{3+\alpha,\infty}\| \rho_{1,\kappa} - \rho_{2,\kappa}\|_{2+\alpha,\infty}  
+ \|\rho_{2,\kappa} \|_{2+\alpha,\infty}\| \rho_{1,\kappa} - \rho_{2,\kappa}\|_{3+\alpha,\infty}\\
\label{S3: SDF est 2}
& + \|\rho_{2,\kappa} \|_{2+\alpha,\infty} \|\rho_{2,\kappa} \|_{3+\alpha,\infty}\| \rho_{1,\kappa} - \rho_{2,\kappa}\|_{1+\alpha,\infty} \Big) \\
\notag + & \, \tilde{C} \Big[ ( \|\rho_{1,\kappa} \|_{2+\alpha,\infty} + \|\rho_{2,\kappa} \|_{2+\alpha,\infty} )\| \rho_{1,\kappa} - \rho_{2,\kappa}\|_{1+\alpha,\infty}  \\
\label{S3: SDF est 3}
& +  \| \rho_{1,\kappa} - \rho_{2,\kappa}\|_{2+\alpha,\infty} + \|\rho_{2,\kappa} \|_{2+\alpha,\infty}  \| \rho_{1,\kappa} - \rho_{2,\kappa}\|_{1+\alpha,\infty} \Big] \\
\notag + &  \, \tilde{C}  \Big( \|\rho_{1,\kappa} \|^2_{2+\alpha,\infty}   \| \rho_{1,\kappa} - \rho_{2,\kappa}\|_{1+\alpha,\infty}  +  \|\rho_{1,\kappa} \|_{2+\alpha,\infty}  \| \rho_{1,\kappa} - \rho_{2,\kappa}\|_{2+\alpha,\infty}    \\
& + \|\rho_{2,\kappa} \|_{2+\alpha,\infty}   \| \rho_{1,\kappa} - \rho_{2,\kappa}\|_{2+\alpha,\infty}  
\label{S3: SDF est 4}
+ \|\rho_{2,\kappa} \|^2_{2+\alpha,\infty}   \| \rho_{1,\kappa} - \rho_{2,\kappa}\|_{1+\alpha,\infty} \Big)\\
\notag + &  \, \tilde{C}  \Big[( \|\rho_{1,\kappa} \|_{2+\alpha,\infty} +\|\rho_{2,\kappa} \|_{2+\alpha,\infty})^2  \| \rho_{1,\kappa} - \rho_{2,\kappa}\|_{2+\alpha,\infty} \\
\label{S3: SDF est 5}
&+ \|\rho_{2,\kappa} \|_{2+\alpha,\infty}^3 \| \rho_{1,\kappa} - \rho_{2,\kappa}\|_{1+\alpha,\infty}\Big].
\end{align}

the definitions of the spaces $E_0, E_\mu, E_\gamma,$ and $E_1,$ 
we have $\mu = 1/4$ and $\gamma = 3/4$ in our current setting. Thus, we refer back to 
\eqref{S2: structure reg 1} to see that index $j$ is subcritical when $(\varrho_j, \gamma_j)$
satisfies $\varrho_j/2+\gamma_j < 1$, and $j$ is critical when $\varrho_j/2+\gamma_j = 1$.
It follows from Proposition~\ref{S3: Prop: retraction of bc} that \eqref{S3: SDF est 1} is bounded by
\begin{align*}
\| \rho_{1,\kappa} - \rho_{2,\kappa}\|_{3+\alpha,\infty}   \leq \sup\limits_{\eta\in \K} \| \rho_{1,\eta} - \rho_{2,\eta}\|_{3+\alpha,\infty} \leq \tilde{C}  \|\rho_1 - \rho_2 \|_{E_\gamma}.
\end{align*}
This corresponds to $(\varrho_j,\gamma_j)=(0, 3/4)$, which is subcritical. 
In \eqref{S3: SDF est 2}, similarly it holds
\begin{align*}
 \|\rho_{1,\kappa} \|_{3+\alpha,\infty}\| \rho_{1,\kappa} - \rho_{2,\kappa}\|_{2+\alpha,\infty}  
\leq \tilde{C} \|\rho_1 \|_{E_\gamma}\| \rho_1 - \rho_2 \|_{2+\alpha,\infty}  
\end{align*}
This corresponds to $(\varrho_j,\gamma_j)=(1, 1/2)$ (which is again subcritical). 
We can estimate the remaining terms of \eqref{S3: SDF est 2} by using 
Propositions~\ref{S3: Prop: retraction of bc} and \ref{S3: Prop: interpolation of bc} 
\begin{align*}
\|\rho_{2,\kappa} &\|_{2+\alpha,\infty}\| \rho_{1,\kappa} - \rho_{2,\kappa}\|_{3+\alpha,\infty}
+  \|\rho_{2,\kappa} \|_{2+\alpha,\infty} \|\rho_{2,\kappa} \|_{3+\alpha,\infty}\| \rho_{1,\kappa} - \rho_{2,\kappa}\|_{1+\alpha,\infty} \\
\leq & \, \tilde{C} \Big( \|\rho_2\|^{1/2}_{E_\gamma}\| \rho_1 - \rho_2\|_{E_\gamma} +   \|\rho_2\|^{3/2}_{E_\gamma}\| \rho_1 - \rho_2\|_{E_\mu}  \Big)
\end{align*}
These correspond to $(\varrho_j,\gamma_j)=(1/2, 3/4)$ and  $(\varrho_j,\gamma_j)=(3/2, 1/4)$, which are critical. 
The remaining terms, i.e. \eqref{S3: SDF est 3}-\eqref{S3: SDF est 5}, can be estimated similarly, 
cf. \cite[Section~4]{LeCroneSim18}. We conclude that \eqref{S3: SDF est 3}-\eqref{S3: SDF est 5} 
is bounded by
\begin{align*}
\tilde{C} & \, \Big[ \big(\|\rho_1\|^{1/2}_{E_\gamma} +  \|\rho_2\|^{1/2}_{E_\gamma} \big)
	\|\rho_1-\rho_2\|_{E_\mu} +\|\rho_1-\rho_2\|_{2+\alpha,\infty}  \\ 
&  + \big(\|\rho_1\|_{E_\gamma}+ \|\rho_2\|_{E_\gamma} \big) \|\rho_1-\rho_2\|_{E_\mu} + 
	\big( \|\rho_1\|^{1/2}_{E_\gamma} + \|\rho_2\|^{1/2}_{E_\gamma} \big)
	\|\rho_1-\rho_2\|_{2+\alpha,\infty}\\
& + \big(\|\rho_1\|_{E_\gamma} + \|\rho_2\|_{E_\gamma} \big) \|\rho_1-\rho_2\|_{2+\alpha,\infty} + \|\rho_2\|^{3/2}_{E_\gamma}\|\rho_1-\rho_2\|_{E_\mu} \Big].
\end{align*}
The indices for those estimates are $(\varrho_j,\gamma_j)=(1/2, 1/4),  (0, 1/2), (1, 1/4),  (1/2, 1/2)$
(subcritical), and $(\varrho_j,\gamma_j)=(1, 1/2), (3/2, 1/4)$ (critical), respectively.

Additionally, it follows from \eqref{S2: infnty,uf} 
that $ \Rc F(\rho) \in  l_{\infty,\uf} (\boldsymbol{bc}^\alpha)$ for any 
$\rho\in V_\mu\cap E_\gamma$. 
Since the constant $\tilde{C}$ is independent of $\kappa$, the above computations together 
with \eqref{S3: SDF est for Q} imply that $F$ satisfies \textbf{(H2)} and
$$
	F\in C^\omega(V_\mu\cap E_\gamma, E_0).
$$

Combining the above discussions, we apply Theorem~\ref{S2: Thm MR} to 
produce the following well--posedness result for \eqref{S3: SDF-rho}.
Note that we assume throughout that $\Sigma$ carries the metric induced by the Euclidean 
metric $g_{m+1}$.

\begin{theorem}
\label{S3: Thm SDF well-posedness}
Let $\alpha\in (0,1)$, $\mu=1/4$ and $ \Sigma$ be a (URT)--hypersurface in $\R^{m+1}$
with a tubular neighborhood of radius $\a$. 
\begin{itemize}
\item[(a)] 
Then for any
$
\rho_0\in V_\mu := \{\rho\in bc^{1+\alpha}(\Sigma): \, \|\rho\|_\infty <\a\},  
$
\eqref{S3: SDF-rho} has a unique solution  
$$
\rho(\cdot,\rho_0)\in C_{1-\mu}(J, bc^{4+\alpha}(\Sigma)) \cap C^1_{1-\mu}(J, bc^\alpha(\Sigma))
$$
on a maximal interval $J=[0,T)=[0,T(\rho_0)),$ with the additional property that
$
\rho(\cdot,\rho_0)\in C(J, bc^{1+\alpha}(\Sigma)).
$

\item[(b)] 
$$
\mathcal{M}:=\bigcup_{t\in(0,T)}(\{t\}\times\Gamma(t))
$$
is a $C^\infty$--hypersurface in $\R^{m+2}$. In particular, each manifold $\Gamma(t)$ is 
$C^\infty$ for $t\in(0,T)$. 
If, in addition, $\Sigma$ is  $C^\omega$--uniformly regular, then 
$\mathcal{M}$ is a $C^\omega$--hypersurface in $\R^{m+2}$.
\vspace{1mm}
\item[(c)]
The map $[(t, \rho_0) \mapsto \rho(t,\rho_0)]$ defines a semiflow on $V_{\mu}$ which is
analytic for $t > 0$ and Lipschitz continuous for $t \ge 0$.
\end{itemize}

\end{theorem}
\begin{proof}
We have already proved part (a) above. Part (b) follows directly from the argument in 
\cite[Sections~3 and 5]{Shao15}. 
For part (c), we first note that Lipschitz continuity of the semiflow follows
from \cite[Corollary~2.3]{LeCroneSim18}. Regarding additional regularity of the semiflow;
for any $\tau > 0$, we note that
$$
\rho(\tau, \rho_0) \in V_\gamma = bc^{3+\alpha}(\Sigma) \cap [\| \rho \|_{\infty} < \a],
$$
and so the result holds in $V_\gamma$ because of \cite[Theorem~6.1]{CS01} and the mapping properties 
of $A(\cdot)$ and $F(\cdot)$. Regularity of the semiflow in $V_\mu$ then follows by embedding.

\end{proof}

\section{The Willmore flow}\label{Section 5}
In this section, we take $\Sigma$ to be a (URT)--hypersurface in  $\R^3$.
For the Willmore flow, we seek a family of hypersurfaces $\{ \Gamma(t) : t \ge 0 \}$
satisfying the evolution equation
\begin{equation}
\label{S5: WF}
\left\{ \begin{aligned}
	V(t) &= -\Delta_{\Gamma(t)}H_{\Gamma(t)} - 2H_{\Gamma(t)} 
		\big( H_{\Gamma(t)}^2 - K_{\Gamma(t)}\big) , \quad t > 0,\\
	\Gamma(0) &= \Gamma_0,
\end{aligned}\right.
\end{equation}
where the term $K_{\Gamma(t)}$ denotes Gaussian curvature of $\Gamma(t)$.

Working in the same setting as Section~\ref{Section 4} above, we consider \eqref{S5: WF}
acting on surfaces $\Gamma(t) = \Gamma_{\rho(t)}$ defined over $\Sigma$ via height
functions $\rho(t) : \Sigma \to \R.$ 
Assuming that $\Sigma$ has a tubular neighborhood $U_\a$ of radius $\a > 0$, 
we recall that $E_\mu = bc^{1+\alpha}(\Sigma)$ and $E_\gamma = bc^{3 + \alpha}(\Sigma)$
are interpolation spaces between $E_0 := bc^{\alpha}(\Sigma)$ and $E_1 := bc^{4 + \alpha}(\Sigma)$,
and we consider initial functions from $V_\mu := \{ \rho \in E_\mu : \| \rho \|_\infty < \a \}$.

Treating \eqref{S5: WF} as a lower--order perturbation of \eqref{S1: SDF}, we again
define
\[
	A(\rho): E_1 \to E_0 : \quad
	[u \mapsto  \frac{1}{m} \ev( g^*(\rho)\otimes  g^*(\rho), \nabla^4 u)]
\]
for all $\rho \in V_\mu,$
and we introduce the mapping $Q: V_\mu \cap E_\gamma \to E_0$
defined as
\[
\begin{split}
	Q(\rho) &:= A(\rho) \rho - \frac{1}{\beta(\rho)} 
		\Psi_\rho^* \Big(\Delta_{g_3|\Gamma_\rho} H_{\Gamma_\rho} + 2 H_{\Gamma_\rho} 
		\big(	H_{\Gamma_\rho}^2 - K_{\Gamma_\rho} \big) \Big)\\
		&= A(\rho)\rho - \frac{1}{\beta(\rho)} \big( \Delta_\rho H_\rho + 
		2H_\rho ( H_{\rho}^2 - K_{\rho})\big).
\end{split}
\]
We thus arrive at the following expression for \eqref{S5: WF} in our current setting:
\begin{equation}
\label{S5: WF-rho}
	\left\{\begin{aligned}
	\partial_t \rho +A(\rho)\rho &= Q(\rho)  &&\text{in}&& (0,\infty)\times \Sigma, \\
	\rho(0) &=\rho_0  &&\text{in }&&\Sigma .&&
	\end{aligned}\right.
\end{equation}
By Proposition~\ref{S3: Prop P(rho) MR}, we know that 
$A \in C^{\omega}(V_\mu, \mathcal{M}_\mu(E_1,E_0))$
so we focus on showing regularity and structural properties for $Q(\rho)$.

By \eqref{S4: Ar Fr} and the definition of $Q(\rho),$ we note that
\[
	Q(\rho) = F(\rho) - \frac{2}{\beta(\rho)} H_\rho^3 + \frac{2}{\beta(\rho)} H_\rho K_\rho 
\]
and it follows that the local expression for $Q(\rho)$ is of the form
\[
	Q(\rho) = 
		\sum\limits_{|\eta|=3,|\tau|\leq 2} 
		c_{\eta,\tau}(\rho,\partial \rho) \, \partial^\tau \rho \, \partial^\eta \rho 
 		+ \sum\limits_{|\eta|,|\sigma|,|\tau|\leq 2} d_{\eta,\sigma,\tau}(\rho,\partial \rho) \,
 			 \partial^\eta \rho \, \partial^\sigma \rho \, \partial^\tau \rho .
\]
To confirm this local expression for $Q(\rho),$ we first note that all third--order
derivatives of $\rho$ appear in $F(\rho)$, while the terms $\frac{2}{\beta(\rho)} H_\rho^3$ 
and $\frac{2}{\beta(\rho)} H_\rho K_\rho$ depend only on up to second--order derivatives.
With the structure for $F(\rho)$ already established in \eqref{S4: F structure}, it suffices to
confirm that $Q(\rho)$ only contributes additional terms of the form
\[
	\sum\limits_{|\eta|,|\sigma|,|\tau|\leq 2} d_{\eta,\sigma,\tau}(\rho,\partial \rho) \,
 			 \partial^\eta \rho \, \partial^\sigma \rho \, \partial^\tau \rho .
\]

Local expressions for $\beta(\rho)$ and $H_\rho$ are given in \eqref{a(rho)-beta(rho)} and 
\eqref{S3: H_rho expression}, respectively.
Since $\beta(\rho)$ depends on at most first--order derivatives of $\rho$ and 
$H_\rho$ depends linearly on second--order derivatives, 
we see that at most cubic powers of $\partial^2 \rho$ appear in 
$\frac{2}{\beta(\rho)} H_\rho^3$. 

Regarding the term $(2 / \beta(\rho)) H_\rho K_\rho,$ we first express Gaussian curvature
\[
	K_\rho = \text{det}[g^{ki}(\rho) l_{ij}(\rho)],
\]
as derived in \cite[Section~2]{Shao13}.
Here $l_{ij}(\rho)$ are the components of the pull--back of the second fundamental 
form of $\Gamma_\rho$. It follows from~\eqref{S3: H_rho expression} that

\begin{align*} \label{S3: L_rho expression}
l_{ij}(\rho) &= \beta(\rho)\Big\{ \partial_i\partial_j\rho
+ ( l^k_j \partial_i \rho-\Gamma^k_{ij})\partial_k \rho   \\
& + \big[r_k^l (\rho)l^k_i \partial_j \rho +
r_k^l(\rho)(\partial_j l_i^k + \Gamma^k_{jh} l_i^h -\Gamma^h_{ij} l_h^k)\rho 
+r_k^l(\rho) l^h_j l^k_h \rho\partial_i \rho \big]\Big\} \partial_l \rho \notag\\
& + \beta(\rho)  ( l_{ij} - l_{ik} l^k_j \rho). \notag
\end{align*}
Observing that each   $l_{ij}(\rho)$ is linear with respect to $\partial^2 \rho$,
it follows that $\partial^2 \rho$ appears at most quadratically in 
$\text{det}[l_{ij}(\rho)],$
since it is a $2 \times 2$ matrix.
Therefore, we conclude that $K_\rho$ contains at most 
quadratic factors of $\partial^2 \rho$ and thus, multiplying with the  
second--order quasilinear term $H_\rho$, 
we conclude that the term $(2 / \beta(\rho)) H_\rho K_\rho$ contains at most cubic powers of 
$\partial^2 \rho.$

With confirmation that $Q(\rho)$ satisfies the same structural condition \eqref{S2: structure reg 2}
as $F(\rho)$ in Section~\ref{Section 4}, we employ the same argument outlined in 
\eqref{S3: SDF est for Q}--\eqref{S3: SDF est 5} to conclude that $(A,Q)$ satisfies 
conditions {\bf (H1)--(H2)}. The following well--posedness result for \eqref{S5: WF}
then follows from Theorem~\ref{S2: Thm MR}.

\begin{theorem}
\label{S5: Thm WF well-posedness}
Let $\alpha\in (0,1)$, $\mu=1/4$ and $\Sigma$ be a (URT)--hypersurface in $\R^3$
with tubular neighborhood of radius $\a$.
\begin{itemize}
\item[(a)] Then for any
$
	\rho_0\in V_\mu := \{\rho\in bc^{1+\alpha}(\Sigma): \, \|\rho\|_\infty <\a\},  
$
\eqref{S5: WF-rho} has a unique solution  
$$
	\rho(\cdot,\rho_0)\in C_{1-\mu}(J, bc^{4+\alpha}(\Sigma)) \cap C^1_{1-\mu}(J, bc^\alpha(\Sigma))
$$
on a maximal interval $J=[0,T)=[0,T(\rho_0)),$ with the additional property that
$
	\rho(\cdot,\rho_0)\in C(J, bc^{1+\alpha}(\Sigma)).
$
\item[(b)]
$$
\mathcal{M}:=\bigcup_{t\in(0,T)}(\{t\}\times\Gamma(t))
$$
is a $C^\infty$--hypersurface in $\R^4$. In particular, each manifold $\Gamma(t)$ is 
$C^\infty$ for $t\in(0,T)$. 
If, in addition, $\Sigma$ is $C^\omega$--uniformly regular, then 
$\mathcal{M}$ is a $C^\omega$--hypersurface in $\R^{m+2}$.
\vspace{1mm}
\item[(c)]
The map $[(t, \rho_0) \mapsto \rho(t,\rho_0)]$ defines a semiflow on $V_{\mu}$ which is 
analytic for $t > 0$ and Lipschitz continuous for $t \ge 0$.
\end{itemize}
\end{theorem}
\begin{proof}
Part (b) follows from \cite{Shao13} and \cite[Section~3]{Shao15}.
Part (c) follows exactly as in the proof of Theorem~\ref{S3: Thm SDF well-posedness}(c) above.
\end{proof}

\subsection{Stability of spheres}

In the case $\Sigma$ is a Euclidean sphere in $\R^3,$ we apply the generalized
principle of linearized stability (c.f. \cite[Section~3]{LeCroneSim18}) to prove the 
following result regarding stability of spheres under the Willmore flow, with control on
only first--order derivatives of perturbations.

\begin{theorem}
\label{S5: Thm WF stability}
Fix $\alpha \in (0,1)$, $\mu = 1/4$, and $\bar{\mu} \in (0,1)$, and let $\Sigma$ be a Euclidean sphere
in $\R^3$ with radius $r > 0$. There exists a constant $\delta \in (0,r)$ such that,
given any admissible perturbation $\Gamma_{\rho_0}$ for 
$$
	\rho_0 \in V_{\mu,\delta} := \{ \rho \in bc^{1 + \alpha}(\Sigma): \| \rho \|_{\infty} < r \text{ and }
		\| \rho_0 \|_{1 + \alpha, \infty} < \delta \},
$$
the solution $\rho(\cdot, \rho_0)$  of~\eqref{S5: WF-rho} exists globally in time and converges to some
$\bar{\rho} \in \mathcal{M}_{sph}$ at an exponential rate, in the topology of $E_{\bar\mu}$.
Here, 
$\mathcal{M}_{sph}$ denotes the 
family of functions $\rho \in C^{\infty}(\Sigma,\R)$
for which $\Gamma_\rho$ is a sphere that is close to $\Sigma$ in $\R^3$.
\end{theorem}

\begin{proof}
It is shown in the proof of \cite[Theorem~1.2]{Sim01} that 
$\rho_* = 0$ is normally stable under \eqref{S5: WF-rho}.
The result then follows from \cite[Theorem~3.2]{LeCroneSim18}.
\end{proof}

\begin{cor}\label{non-convex}
	There exist non--convex hypersurfaces $\Gamma_0$ such that the solution $\rho(\cdot,\rho_0)$ to
	\eqref{S5: WF-rho} with $\Gamma(\rho_0) = \Gamma_0,$ exists globally in time and 
	converges exponentially fast to a sphere.
\end{cor}
We note here that Theorem~\ref{S5: Thm WF stability} also holds true for the surface diffusion flow, 
as was shown in~\cite[Section~4.5]{LeCroneSim18}.

\appendix
\section{}
Suppose $\Sigma$ is a (URT)-hypersurface with tubular neighborhood of radius $\a$.
Given $\rho \in C(\Sigma)$ with $\| \rho \|_{\infty} < \a$, let
$\Gamma_\rho:=\Psi_\rho(\Sigma),$ where $ \Psi_\rho( \p)=\p + \rho( \p) \nu_{\Sigma}(\p)$ for $\p\in\Sigma$.

Then $\Gamma_\rho$ enjoys the following properties.
\begin{prop}\label{A-Gamma_rho URT}
 Let $k\in\N\cup\{\infty\}$.
\begin{itemize}
\item[(a)]  Suppose $\rho\in BC^{k+1}(\Sigma)$  and $\|\rho\|_{\infty}<\a$. 
Then $\Gamma_\rho$ is $C^k$--uniformly regular.
\vspace{1mm}
\item[(b)]
There exists   $\varepsilon_1>0$ such that, for any $\rho\in BC^{2}(\Sigma)$ with $\|\rho\|_\infty\le \varepsilon_1$, the hypersurface $\Gamma_\rho$ has a tubular neighborhood of radius $\a_1$ for some positive number $\a_1=\a_1(\varepsilon_1,\rho)$.
\end{itemize}
\end{prop}
\begin{proof} 
(a) 
We can construct an atlas $\A_\rho=\{{\sf O}_{\kappa,\rho}, \varphi_{\kappa,\rho}): \kappa\in \K\}$ 
for $\Gamma_\rho$ as follows. 
Define
$$
{\sf O}_{\kappa,\rho}:=\Psi_\rho \circ \psi_\kappa (\Bm),\quad  \varphi_{\kappa,\rho}:= \varphi_\kappa \circ \Psi_\rho^{-1},\quad \psi_{\kappa,\rho}= \varphi_{\kappa,\rho}^{-1}.
$$ 
Then $\A_\rho$ inherits properties (R1)--(R3) from $\A$.
Next we note that 
$$\psi_{\kappa,\rho}^* (g_{m+1}|_{\Gamma_\rho})= \psi_\kappa^*g(\rho),$$
and  that by  \eqref{g(rho)-ij}, $g(\rho)$ involves first order derivatives of $\rho$.
Hence,  $\psi_\kappa^*g(\rho)$ is $C^k$ for $\rho\in C^{k+1}$.
It follows readily from \eqref{g(rho)-ij} that 
\begin{equation*}
\begin{aligned}
|(I - \rho L_\Sigma)\xi |^2 (\p)\le g_{ij}(\rho)(\p)(\xi,\xi)\le (1+|a(\rho)|^2) | (I-\rho L_\Sigma)\xi|^2(\p)
\end{aligned}
\end{equation*}
for $\p\in\Sigma$ and $\xi=\xi^i\tau_i(\p)\in T_\p \Sigma$.
Properties (R4)-(R5) now follow from
$$\min\{(1-(\rho\kappa_r)(\p))^2\}|\xi|^2\le |(I-\rho L_\Sigma)\xi|^2(\p)\le \max\{(1-(\rho\kappa_r)(\p))^2\}|\xi|^2$$
and Remark~\ref{S3: Rmk: tubular-nbgh}(d).

\medskip\noindent
(b) 
Let $r_0\in (0,1)$ be the constant related to the uniformly shrinkable property.

\medskip\noindent
{\bf Claim 1:} 
Let $\tilde{r}_0:=\frac{1+r_0}{2}$.
There exists a uniform constant $r_1$ such that for any $\kappa\in \K$ and $\p \in \psi_\kappa(\tilde{r}_0\Bm)$, 
$\psi_\kappa(\B^m( x_\p, r_1))$ is a graph $f_{\kappa,\p}$ over $T_\p \Sigma$ with $x_\p := \varphi_\kappa(\p)$ satisfying 
\begin{equation}
\label{f-bound}
\|f_{\kappa,\p}\|_{2,\infty}\leq c_0
\end{equation}
for some $c_0>0$ independent of $\kappa$ and $\p$.
\begin{subproof}[Proof of Claim 1] 
Let $\kappa\in\K$ and $\p\in \psi_\kappa(\tilde r_0\B^m)$ be given.
 Then there exists $x_{\p}$ in $\tilde{r}_0\Bm$ such that $\p=\psi_k(x_{\p})$.
Let 
$$
{\mathcal P}_\p:=I -\nu_{\Sigma}(\p)\otimes\nu_{\Sigma}(\p) 
$$ 
be the   orthogonal projection of $\R^{m+1}$ onto $T_\p\Sigma$. 
In the following we will identify $T_\p\Sigma$ with $\R^m$.
It follows from the boundedness of $\|L_\Sigma\|_\infty$ that
there is a universal constant $b_0$ such that
${\mathcal P}_\p: \psi_\kappa(\B(x_\p,b_0)) \to T_\p(\Sigma) \doteq \R^m$ is injective.
Let
$$
F_{\kappa,\p}(x):=({\mathcal P}_\p \circ\psi_\kappa) (x),\quad x\in \B(x_\p,b_0).
$$
Then we obtain  for the Fr\'echet derivative of $F_{\kappa,\p}$
\begin{equation}
\label{DF}
DF_{\kappa,\p} (x_\p)={\mathcal P}_\p D\psi_\kappa(x_\p)= D\psi_\kappa(x_\p),
\end{equation}
as $D\psi_\kappa(x_\p)\xi\in T_\p \Sigma $ for all $\xi\in \R^m$.
We infer from (R5) that
\begin{equation}
\label{D psi}
 (1/\gamma_1)^2 |\xi|^2\leq |D\psi_\kappa (x) \xi|^2=   (\psi_\kappa^* g)(x)(\xi,\xi)  \leq  \gamma^2_1 |\xi|^2, \
\quad  x\in \B^m, \;\; \xi\in\R^m,
\end{equation}
for some uniform constant $\gamma_1\geq 1$.
It follows from  \eqref{DF} and \eqref{D psi} that
the spectrum of $DF_{\kappa,\p}(x)$ lies outside the ball $\B_{\mathbb C}(0,1/\gamma_1)$ for any $x\in\B^m$.
Indeed, suppose
$\mu v = DF_{\kappa,p}(x) v$
for some $\mu\in \C$ and $v=\xi +i\eta \in \C^m$ with $|v|=1$.
Then 
$$
|\mu|^2 
=|DF_{\kappa,p}(x) v|^2
=|D\psi_\kappa(x)\xi|^2 + |D\psi_\kappa(x)\eta|^2\ge (1/\gamma_1)^2.
$$
Lemma 4.1 in \cite{AHS94} 
implies that $DF_{\kappa,\p}(x)$ is invertible with 
\begin{equation}
\label{DF-inverse}
|[DF_{\kappa,\p}(x)]^{-1}|\le \gamma_2,
\end{equation}
where the constant $\gamma_2$ is independent of $x\in\B^m$ and $\kappa,\p$.
By the inverse function theorem, there exists a uniform constant $r_1$
which is independent of $\kappa$ and $\p\in\psi_\kappa(\tilde{r}_0\Bm)$
such that
$$
F_{\kappa,\p}: \B^m(x_\p, r_1)\to {\mathcal P}_\p\psi_\kappa(\B^m(x_\p, r_1))
$$
is a diffeomorphism.
Next we note that 
\begin{equation*}
\begin{aligned}
\partial_j F^{-1}_{\kappa,\p}(y) &= [DF_{\kappa,\p}(F^{-1}_{\kappa,\p}(y))]^{-1}e_j \\
\partial_i\partial_j F^{-1}_{\kappa,\p}(y)
&=-[DF_{\kappa,\p}(F^{-1}_{\kappa,\p}(y))]^{-1}\;\partial_i  [DF_{\kappa,\p}(F^{-1}_{\kappa,\p}(y))]\;
 [DF_{\kappa,\p}(F^{-1}_{\kappa,\p}(y))]^{-1}e_j
\end{aligned}
\end{equation*}
Recall that
$
\partial_i\partial_j \psi_\kappa =\Gamma_{ij}^k \partial_k \psi_\kappa + l_{ij} \nu_\Sigma.
$
In view of \cite[Formula~(3.19)]{Ama13}, \eqref{D psi} and the boundedness of $\|L_\Sigma\|_\infty$, we conclude that
\begin{equation}
\label{D2 psi}
\| \psi_\kappa  \|_{2,\infty} \leq \gamma_3, \qquad \text{for all $\kappa \in \K$}.
\end{equation}

\medskip
It follows from   \eqref{DF-inverse} and \eqref{D2 psi} that 
$\|F_{\kappa,\p}^{-1}\|_{2,\infty}\leq c$ for some $c$ independent of $\kappa,\p$.
Define  $\Phi_{\kappa,\p}:   {\mathcal P}_\p \psi_\kappa(\B^m(x_\p, r_1))
 \to \Sigma$ 
by $\Phi_{\kappa,\p}:=\psi_\kappa\circ F^{-1}_{\kappa,\p}.$
Note that
\begin{align*}
\Phi_{\kappa,\p}(y)
\notag&={\mathcal P}_\p \circ \Phi_{\kappa,\p}(y) + (I-{\mathcal P}_\p) \circ \Phi_{\kappa,\p}(y) \\
\label{graph function}&= y + (\nu_\Sigma(\p)|\Phi_{\kappa,\p}(y))\nu_\Sigma(\p))
=: y +f_{\kappa,\p}(y)\nu_\Sigma(p).
\end{align*}
We can now conclude that \eqref{f-bound} holds.
\end{subproof}

 In the following, we assume that 
\begin{equation}
\label{r1-small}
r_1<\min\Big\{\frac{1-r_0}{2}, \frac{1}{2\gamma_1 \gamma_3}\Big\}.
\end{equation}
By Claim 1, we can find $L\in \N$ such that, in every $\Uk$, there exist $x_{\kappa,i} \in r_0\bar{\B}^m$ with $i=1,\cdots,L$ such that 
\begin{center}
$\bigcup_{i=1}^L \psi_\kappa(\Bm(x_{\kappa,i}, r_1/4))$ covers $\psi_\kappa(r_0\bar\B^m)$.
\end{center}
Taking new local patches $\psi_\kappa(\Bm(x_{\kappa,i}, r_1/2))$,
after relabelling, translation and scaling, we obtain a new atlas satisfying (R1)--(R5), still denoted 
by $\A=\{(\Uk,\varphi_\kappa): \kappa\in \K\}$. 
Note that for this new atlas, $\Uk$ is the graph of a function $f_{\kappa,\p}$ over $T_\p \Sigma$
for any $\kappa\in\K$ and $\p\in \Uk$. Moreover, \eqref{f-bound} still holds  true.
In addition, we can take uniformly shrinkable constant $r_0=1/2$.
Note also that, by (R5), we can assume that  $r_1$ is chosen so small that
\begin{equation}
\label{small patch}
|\p - \q|<\a/8,\quad \p,\q \in \Uk.
\end{equation}

Let ${\rm dist}(\cdot,\cdot)$ denote the Euclidean distance between two compact subsets in $\R^{m+1}$.

\medskip\noindent
{\bf Claim 2:} There exists $c_1>0$ such that ${\rm dist}(\p,   \partial  X(\Uk, [-\a/2,\a/2])) 
> c_1$ for all $\kappa\in \K$ and $\p\in \psi_\kappa(\frac{1}{2} \Bm)$.

\begin{subproof}[Proof of Claim 2] 
We set
$$
D_\kappa := X (\Uk\times [-\a/2, \a/2 ]),
$$
and
$$
S_{1,\kappa}:= X (\Uk\times \{-\a/2\}) \cup X (\Uk\times \{\a/2\}), \quad S_{2,\kappa} := \partial D_\kappa \setminus S_{1,\kappa}.
$$
We now  show that ${\rm dist}(\p, \partial D_\kappa)$ is uniformly positive.

\medskip
\noindent
Case 1: $|\p-q|= {\rm dist}(\p, \partial D_\kappa)$ for some $q\in S_{1,\kappa}$

\smallskip
\noindent
Since $q\in S_{1,\kappa}$, we can find some $\q\in \Uk$ such that,  without 
loss of generality, we have $X( \q, \a/2)=q$.
This implies
$$
|\p - q| \geq \a/2 - | \p - \q | \geq \a/4.
$$

\smallskip
\noindent
Case 2: $|\p-q|= {\rm dist}(\p, \partial D_\kappa)$ for some $q\in S_{2,\kappa}.$

\smallskip
\noindent
We first observe that, for any $\p \in \psi_\kappa(\frac{1}{2}(\B^m))$, 
we have by  \eqref{D psi}, \eqref{D2 psi} and Taylor expansion that
\begin{equation*}
|\p -\q|=|\psi_\kappa(x_\p) -\psi_\kappa(x_\q)| 
\ge \frac{r_1}{2}\left[\frac{1}{2\gamma_1} - \frac{r_1}{2} \gamma_3\right],
\end{equation*}
where $\p=\psi_\kappa(x_\p)$ with $x_\p$ in $\frac{1}{2}\B^m$ and $\q=\psi_\kappa(x_\q)$ with $x_\q\in \partial\B^m$.
The assumption \eqref{r1-small} now implies
\begin{equation}
\label{distance}
{\rm dist}(\p, \partial \Uk)\geq \eta_0=r_1/(8\gamma_1).
\end{equation}
\medskip
Moreover, $S_{2,\kappa}= X(\partial \Uk \times (-\a/2, \a/2))$ implies that there exist $\q\in \partial \Uk $ and $s\in (-\a/2, \a/2)$ such that $q=X(\q,s)$.

Because of \eqref{small patch}, we can  realize $\p$ as a point on the graph 
of $f_{\kappa,\q} $, cf. the following figure.

\begin{center}
\includegraphics[scale=.8]{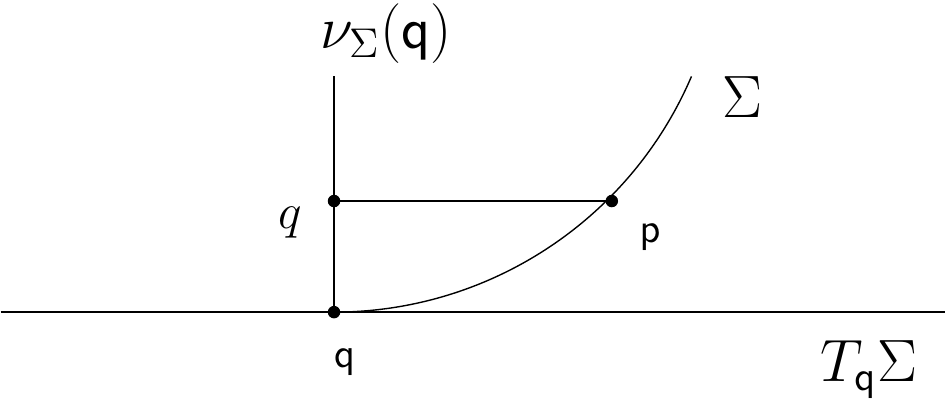}
\end{center}

By \eqref{distance}, we observe that  $|\p-\q|\geq \eta_0$. 
Let  $d=|\p-q|$.  Using \eqref{f-bound}, we have
$$
d^2 + d^2 c_0^2 \geq \eta_0^2,
$$
which implies
$$
|\p-q|\geq \eta_1
$$
for some uniform constant $\eta_1>0$.
Thus we can take $c_1=\min\{\a/4,\eta_1\}$ independent of  $\p$ and $\kappa$.
\end{subproof}

\medskip\noindent
{\bf Claim 3:} 
Let $\delta\in (0,\a)$ and $M>0$ be fixed.
Then there exists a constant $r_2>0$ such that for any  $\rho\in BC^2(\Sigma)$ with $\|\rho\|_\infty\le \delta$, $\|\rho\|_{2,\infty}\le M,$
 and any $\p\in \psi_\kappa(\frac{3}{4} \Bm)$,
 $\Psi_\rho(\psi_\kappa(x_\p, r_2))$ is the graph of a $C^2$-function $h_{\kappa,\p}$ over $T_{\Psi_\rho(\p)}\Gamma_\rho$ satisfying
\begin{equation}
\label{h-bound}
\| h_{\kappa,\p}\|_{2,\infty} \leq c_2,
\end{equation}
for some $c_2>0$ independent of $\kappa,\p$ and $\rho$.
\begin{subproof}[Proof of Claim 3] 
The proof is basically the same as that of Claim 1, as $\Gamma_\rho$ is a $C^2$-hypersurface and $C^1$-{uniformly regular} by
part (a) of the proposition; and all we need for the proof of Claim 1  is this property.
\end{subproof}
We  assume   
\begin{equation}\label{ASP}
\varepsilon_1<\min\{\a/8,c_1/2\}, 
\end{equation}
where $c_1$ is the constant in Claim 2.
By our choice of $\varepsilon_1$, following the construction below Claim 1, we can further modify the atlas $\A$, still with 
uniform shrinkable parameter $r_0=1/2$, such that for every 
$\rho\in BC^2(\Sigma)$ with $\|\rho\|_\infty\le \varepsilon_1$,
$\Psi_\rho(\Uk)$ is a graph over $T_{\Psi_\rho(\p)}\Gamma_\rho$,
for any $\kappa\in\K$ and $\p\in\Uk$.

Claim 3 and Example 3.2(c) imply that $\Psi_\rho(\Uk)$ has a tubular neighborhood of radius $\a_1$, 
where $\a_1$ is independent of $\kappa,\p$.
In order to prove that $\Gamma_\rho$ has a tubular neighborhood of radius $\a_1$, 
it suffices to show that 
$$
X_\rho: \Sigma \times (-\a_1,\a_1): (\p,s)\mapsto \Psi_\rho(\p)+ s \nu_{\Gamma_\rho}(\Psi_\rho(\p))
$$
is injective.

\medskip\noindent
{\bf Claim 4:} For sufficiently small $\a_1>0$ and any $\rho\in  BC^2(\Sigma)$ with 
$\|\rho\|_{\infty} \leq \varepsilon_1$ satisfying \eqref{ASP} and 
$\p\in \psi_\kappa(\frac{1}{2}\Bm)$, it holds that $\B_{\R^{m+1}}(\Psi_\rho(\p),2\a_1)$ 
is contained in $X(\Uk,[-\a/2,\a/2])$.
\begin{subproof}[Proof of Claim 4] 
Define $D_\kappa$, $S_{1,\kappa}$ and $S_{2,\kappa}$ as in Claim 2. 
Given any $\p \in \psi_\kappa(\frac{1}{2}\Bm)$, there exists some $q\in \partial D_\kappa$ such that
$$
| \Psi_\rho (\p) - q |= {\rm dist}(\Psi_\rho (\p), \partial D_\kappa).
$$
If $q\in S_{1,\kappa}$, then there exists $\q \in \Uk$ so that,
without loss of generality, $X(\q,\a/2)=q$. 
By \eqref{small patch} and \eqref{ASP}, we infer that
$$
| \Psi_\rho (\p) - q | \geq \a/2 - |\Psi_\rho(\p)-\p| - | \p -\q| \geq  \a/4.
$$
If $q\in S_{2,\kappa}$, by Claim 2 and \eqref{ASP}
$$
| \Psi_\rho (\p) - q | \ge |\p-q| - |\rho(\p)| \geq c_1/2.
$$
Therefore, it suffices to take $\a_1\leq \min\{\a/8, c_1/4\}=c_1/4$.
\end{subproof}

If $X_\rho(\p,s)= X_\rho(\q,t)$ for some $\p, \q\in \Sigma$ and $s,t \in  (-\a_1, \a_1)$,
we may assume that  $\p\in \psi_\kappa(\frac{1}{2}\Bm)$ for some $\kappa\in\K$.
It follows that 
$$
|\Psi_\rho(\p)-\Psi_\rho(\q)| = |s\nu_\Sigma(\p)-t\nu_\Sigma(\q)|<2 \a_1.
$$ 
We conclude from Claim 4 that 
$\Psi_\rho(\q) \in X(\Uk,[-\a/2,\a/2])$ and thus $\q \in \Uk$ as well.
However, in this case, we already know that $\Psi_\rho(\Uk)$ has an 
$\a_1$--tubular neighborhood, which implies that $\p=\q$ and $s=t$.
\end{proof}

\begin{remark}
In Proposition~\ref{A-Gamma_rho URT}(b) it would be desirable to be able to replace the smallness condition 
$\|\rho\|_\infty \le \varepsilon_1$ by the more natural condition $\|\rho\|_\infty <\a$.

In the special case that $\Sigma$ is compact, this property  holds by Remark~\ref{S3: Rmk: tubular-nbgh}(b), 
as $\Gamma_\rho$ is a compact (closed) $C^2$-hypersurface.
\end{remark}




\end{document}